\documentclass[final]{siamltex}

\usepackage{amsmath, amssymb, algorithm, algpseudocode}
\usepackage{xifthen}
\usepackage{subfigure}
\usepackage{graphicx}
\usepackage{multirow}
\usepackage[usenames,dvipsnames]{color}
\usepackage{soul}
\usepackage[colorlinks=true, citecolor=blue]{hyperref}
\usepackage[shortlabels]{enumitem}
\usepackage{float}
\newfloat{procedure}{tbp}{lop}
\floatname{procedure}{Procedure}

\def\argmin{{\rm argmin}}
\def\Argmin{{\rm Argmin}}
\def\lb{{\rm lb}}
\def\ub{{\rm ub}}

\let\Oldul\underline
\renewcommand{\underline}[1]{\Oldul{\smash{#1}}}

\newcommand{\cE}{\ensuremath{\mathcal{E}}}

\newcommand{\cO}{\ensuremath{\mathcal{O}}}

\newcommand{\R}{\ensuremath{\mathbb{R}}}

\newcommand{\cL}{\ensuremath{\mathcal{L}}}

\def\pi{\ensuremath{\hat{\cal K}_x(\bar x_i)}}

\newcommand{\ag}[1][]{\ensuremath{
	\ifthenelse{\isempty{#1}}
		{_{t}^{ag}}
		{_{t+#1}^{ag}}}}
\newcommand{\md}[1][]{\ensuremath{
		\ifthenelse{\isempty{#1}}
		{_{t}^{md}}
		{_{t+#1}^{md}}}}
\newcommand{\bg}[1][]{\ensuremath{
		\ifthenelse{\isempty{#1}}
		{\beta_{t}\gamma_t}
		{(\beta_{t}-1)\gamma_t}}}
\newcommand{\rst}[1][]{\ensuremath{
		\ifthenelse{\isempty{#1}}
		{\frac 1{2\eta_t}}
		{\frac 1{2\eta_{t+1}}}}}
\newcommand{\rtt}[1][]{\ensuremath{
		\ifthenelse{\isempty{#1}}
		{\frac 1{2\tau_t}}
		{\frac 1{2\tau_{t+1}}}}}
\newcommand{\xdiff}[1][]{\ensuremath{
		\ifthenelse{\isempty{#1}}
		{\frac 1{2\eta_t}\|x-x_t\|^2 - \frac 1{2\eta_t}\|x-x_{t+1}\|^2}
		{\frac {\gamma_t}{2\eta_t}\|x-x_t\|^2 - \frac {\gamma_t}{2\eta_t}\|x-x_{t+1}\|^2}}}
\newcommand{\ydiff}[1][]{\ensuremath{
		\ifthenelse{\isempty{#1}}
		{\frac 1{2\tau_t}\|y-y_t\|^2 - \frac 1{2\tau_t}\|y-y_{t+1}\|^2}
		{\frac {\gamma_t}{2\tau_t}\|y-y_t\|^2 - \frac {\gamma_t}{2\tau_t}\|y-y_{t+1}\|^2}}}

\newcommand{\bbr}{\mathbb{R}}
\def\eqnok#1{(\ref{#1})}

\def\vgap{\vspace*{.1in}}
\newcommand{\beq}{\begin{equation}}
\newcommand{\eeq}{\end{equation}}
\newcommand{\beqa}{\begin{eqnarray}}
\newcommand{\eeqa}{\end{eqnarray}}
\newcommand{\beqas}{\begin{eqnarray*}}
\newcommand{\eeqas}{\end{eqnarray*}}
\title{Fast Bundle-Level Type Methods for unconstrained and ball-constrained convex optimization\thanks{December, 2014.
This research was partially supported by NSF
grants CMMI-1254446, DMS-1319050, and ONR grant N00014-13-1-0036.
}}
\author{Yunmei Chen\thanks{Department of Mathematics, University of Florida ({\tt yun@math.ufl.edu}).} \and
Guanghui Lan \thanks{Department of Industrial and System Engineering, University of Florida ({\tt glan@ise.ufl.edu})}. \and
Yuyuan Ouyang \thanks{Department of Industrial and System Engineering, University of Florida ({\tt ouyang@ufl.edu})} .\and
Wei Zhang \thanks{Department of Mathematics, University of Florida ({\tt weizhang657@ufl.edu})}.
}

\begin{document}

\maketitle

\begin{abstract}
It has been shown in \cite{Lan13-1} that the accelerated prox-level (APL) method and its variant, the uniform smoothing level (USL) method, have optimal iteration complexity for solving black-box and structured convex programming problems without requiring the input of any smoothness information. However, these
algorithms require the assumption on the boundedness of the feasible set and their efficiency relies on the solutions of two involved subproblems. These hindered the applicability of these algorithms in solving large-scale and unconstrained optimization problems. In this paper, we first present a generic algorithmic framework to extend
these uniformly optimal level methods for solving unconstrained problems. Moreover, we introduce two new variants
of level methods, i.e., the fast APL (FAPL) method and the fast USL (FUSL) method, for solving large scale
black-box and structured convex programming problems respectively. Both FAPL and FUSL enjoy the same optimal iteration complexity as APL and USL, while the number of subproblems in each iteration is reduced from two to one. Moreover, we present an exact method to solve the only subproblem for these algorithms.
As a result, the proposed FAPL and FUSL methods have improved the performance of the APL and USL in practice significantly in terms of both
computational time and solution quality. Our numerical results on solving some large-scale least square problems and total variation based image reconstruction have shown great advantages of these new bundle-level type methods over APL, USL, and some other state-of-the-art first-order methods.
\end{abstract}



\setcounter{equation}{0}
\section{Introduction}
\label{secIntro}
Given a convex function $f: \R^n\rightarrow \mathbb{R}$, the main problem of interest
in this paper is:
\begin{equation}
\label{oriproblem}
f^\ast :=\min_{x\in \R^n} f(x).
\end{equation}
Throughout this paper, we assume that
the solution set $X^{\ast}$ of \eqnok{oriproblem} is nonempty.
Moreover, denoting
\begin{equation}
\label{eqnxstar}
x^{\ast}:=\argmin_{x}\{\|\overline{x} - x\|: x\in X^{\ast}\} \ \ \mbox{and} \ \ \ D^{\ast} := \| \overline{x} - x^{\ast}\|
\end{equation}
for a given initial point $\overline{x} \in \bbr^n$,
we assume in the black-box setting that for all closed sets
\begin{equation} \label{def_Ball}
\Omega\subseteq B(\overline{x}, 4D^{\ast}) := \left \{ x \in \bbr^n: \|x - \overline{x}\| \le 4D^{\ast} \right\},
\end{equation}
there exists $M(\Omega)>0$ and $\rho(\Omega)\in [0,1]$, such that\footnote{Observe that
this assumption is weaker than requiring \eqnok{smoothrelation} holds for any
$x,y \in \bbr^n$.}
\begin{equation}
  \label{smoothrelation}
 f(y)-f(x)-\left\langle f'(x),y-x\right\rangle \le\frac{M(\Omega)}{1+\rho(\Omega)}\|y-x\|^{1+\rho(\Omega)}, \ \ \forall x,y \in \Omega.
\end{equation}
Here $\|\cdot\|$ denotes the Euclidean norm, and $f'(x)\in \partial f(x)$, where $\partial f(x)$ denotes the subdifferential of $f$ at $x\in \Omega$.
Clearly, the above assumption on $f(\cdot)$ covers non-smooth ($\rho(\Omega)=0$), smooth ($\rho(\Omega)=1$) and weakly smooth ($0<\rho(\Omega)<1$) functions.

Our main goal in this paper is to develop a bundle-level method that is able to compute an approximate solution to problem \eqref{oriproblem}, with uniformly optimal iteration complexity for non-smooth, smooth and weakly smooth objective functions (see \cite{Lan13-1}). Here, the iteration complexity is described by the number of evaluations of subgradients/gradients of $f$. In addition, we aim to design the bundle-level method so that its
computational cost in each iteration is small, in order to improve its practical performance. Let us start by reviewing a few existing bundle-level methods.

%It has been proved in \cite{nemirovskici1983problem} that if $X$ is compact and convex, in the classical black-box setting, the number of iterations required by any first order method to find an $\epsilon$-solution, which is a point $\overline{x}$ such that $f(\overline{x})-f^{\ast}\le \epsilon$, can not be smaller than ${\cal O}(1/\epsilon^2)$ when $f(\cdot)$ is non-smooth, or ${\cal O}(1/\sqrt{\epsilon})$ when $f(\cdot)$ is smooth.

\subsection{Cutting plane, bundle and bundle-level methods}
\label{secBL}
The bundle-level method originated from the well-known Kelley's cutting-plane method in 1960 \cite{Kelley60}.
Consider the convex programming (CP) problem of
\begin{align}
	\label{eqnProblemX}
	\min_{x\in X}f(x),
\end{align}
where $X$ is a compact convex set and $f$ is a closed convex function. The fundamental idea of the cutting plane method is to generate a sequence of
piecewise linear functions to approximate $f$ in $X$. In particular, given $x_1,x_2,\ldots,x_k\in X$,
we approximate $f$ by
\begin{equation}
\label{convexmodel}
m_k(x):=\max\{h(x_i,x),1\leq i\leq k\},
\end{equation}
and compute the iterates $x_{k+1}$ by:
\begin{equation}
\label{kellymethod}
x_{k+1}\in \Argmin_{x\in X} m_k(x),
\end{equation}
where
\begin{equation}
\label{linearappro}
h(z,x):=f(z)+\left\langle f'(z),x-z\right\rangle .
\end{equation}
Clearly, the functions $m_i$, $i = 1, 2, \ldots$, satisfy
$m_i(x)\leq m_{i+1}(x)\leq f(x)$ for any  $x\in X$, and are identical to $f$ at those search points $x_i$, $i = 1, \ldots, k$.
However, the inaccuracy and instability of the piecewise linear approximation $m_k$
over the whole feasible set $X$ may affect the selection of new iterates, and the above scheme converges slowly both theoretically and practically \cite{nemyud:83,{Nest04}}. Some important improvements of
Kelley's method have been made in 1990s under the name of bundle methods (see, e.g., \cite{Kiw90-1,Kiw95-1,LNN}).
In particular, by incorporating the level sets into Kelley's method,
Lemar\'{e}chal, Nemirovskii and Nesterov~\cite{LNN} proposed in 1995
the classic bundle-level (BL) method by performing a series of projections over
the approximate level sets.

Given $x_1,x_2,\ldots,x_k$, the basic BL iteration consists of the following three steps:
\begin{itemize}
\item [a)] Set $\overline{f}_k:=\min \{f(x_i),1\le i\le k\}$ and compute
a lower bound on $f^*$ by
$
\underline{f}_k = \min_{x\in X} m_k(x).
$
\item [b)] Set the level $l_k = \beta\underline{f}_k + (1-\beta)\overline{f}_k$ for some $\beta\in (0,1)$.
\item [c)] Set $ X_k:=\{x\in X: m_k(x)\le l_k\}$ and determine the new iterate
\beq \label{BLsubp}
x_{k+1} = \argmin_{x \in X_k} \|x - x_k\|^2.
\eeq
\end{itemize}
In the BL method, the localizer $X_k$  is used to
approximate the level set $L_k := \{x: f(x) \le l_k \}$, because
the projection over $L_k$ is often too difficult to compute. Intuitively, as $k$ increases,
the value of $l_k$ will converge to $f^*$, and consequently both $L_k$ and $X_k$ will
converge to the set of optimal solutions for problem \eqnok{eqnProblemX}.
It is shown in \cite{LNN}
the number of BL iterations required to find an $\epsilon$-solution of \eqnok{eqnProblemX},
i.e., a point $\hat x \in X$ s.t. $f(\hat x) - f^* \le \epsilon$, can be bounded by ${\cal O} (1/\epsilon^2)$,
which is optimal for general nonsmooth convex optimization.
%The empirical convergence of the BL method, however, is often much faster, exhibiting a ${\cal O}(\log (1/\epsilon))$
%iteration bound (see \cite{BenNem00}).
% One other significant advantage of
%the BL method exists in that it does not require the input of any problem parameters.

Observe that for the above BL methods, the localizer $X_k$ accumulates constraints,
and hence that the subproblem in step c) becomes more and more expensive to solve.  In order
 to overcome this difficulty, some restricted memory BL algorithms have been developed in \cite{Kiw95-1,BenNem05-1}.
In particular, Ben-Tal and Nemirovski \cite{BenNem05-1}  introduced
the non-Euclidean restricted memory level (NERML) method, in which
the number of extra linear constraints in $X_k$ can be as small as $1$ or $2$,
without affecting the optimal iteration complexity. Moreover,
the objective function $\|\cdot\|^2$ in \eqnok{BLsubp} is replaced by a
general Bregman distance $d(\cdot)$ for exploiting the geometry of the feasible set $X$.
NERML is often regarded as
state-of-the-art  for large-scale nonsmooth convex optimization as it
substantially outperforms subgradient type methods in practice. Some more recent development of inexact proximal bundle methods and BL methods could be found in
\cite{van2014constrained,oliveira2011inexact,oliveira2013bundle,richtarik2012approximate,
kiwiel2006proximal,kiwiel2009bundle,DeSa14-1,kiwiel2009inexact}.

\subsection{Accelerated bundle-level type methods}
\label{secABL}
While the classic BL method was optimal for solving nonsmooth CP problems only, Lan~\cite{Lan13-1} recently
significantly generalized this method so that they can optimally solve any black-box CP problems, including non-smooth, smooth and weakly smooth CP problems. In particular, for problem \eqref{eqnProblemX} with compact feasible set $X$, the two new BL methods proposed in \cite{Lan13-1}, i.e., the accelerated bundle-level(ABL) and accelerated prox-level(APL) methods, can solve these problems optimally without requiring any information on problem parameters. The ABL method can be viewed as an accelerated version of the classic BL method. Same as the classic BL method, the lower bound on $f^*$ is estimated from the cutting plane model $m_k$ in \eqref{convexmodel}, the upper bound on $f^*$ is given by the best objective value found so far, and the prox-center is updated by \eqref{BLsubp}. The novelty of the ABL method exists in that three different sequences, i.e, $\{x_k^l\},\{x_k\}$ and $\{x_k^u\}$, are used for updating lower bound, prox-center, and upper bound respectively, which leads to its accelerated iteration complexity for smooth and weakly smooth problems. The APL method is a more practical, restricted memory version of the ABL method,  which can also employ non-Euclidean prox-functions to explore the geometry of the feasible set $X$.

We now provide a brief description of the APL method. This method consists of multiple phases, and each phase calls a gap reduction procedure to reduce the gap between the lower and upper bounds on $f^*$ by a constant factor. More specifically, at phase $s$, given the initial lower bound $\lb_s$ and upper bound $\ub_s$, the APL gap reduction procedure sets $\underline{f}_0=\lb_s, \overline{f}_0=\ub_s, l=\beta \cdot \lb_s+(1-\beta)\ub_s$,
and
iteratively performs the following steps.
\begin{itemize}
\item [a)] Set $x_k^l=(1-\alpha_k)x_{k-1}^u+\alpha_k x_{k-1}$ and $\underline{f}_k:=\max\{\underline{f}_{k-1},\min\{l,\underline{h}_k\}\}$, where
\begin{equation}
\label{APLsub1}
\underline{h}_k:=\min_{x\in X_{k-1}}h(x_k^l,x).
\end{equation}
\item [b)] Update the prox-center $x_k$ by
\begin{equation}
\label{APLsub2}
x_k=\argmin_{x\in \underline{X}_k} d(x),
\end{equation}
where $\underline{X}_k:=\{x\in X_{k-1}: h(x_k^l,x)\le l\}$.
\item [c)] Set $\overline{f}_k=\min\{\overline{f}_{k-1},f(\tilde{x}_k)\}$, where $\tilde{x}_k^u=\alpha_k x_k+(1-\alpha_k)x_{k-1}^u$, and $x_k^u$ is chosen as either $\tilde{x}_k^u$ or $x_{k-1}^u$ such that $f(x_k^u)=\overline{f}_k$.
\item [d)] Choose $X_k$ such that $\underline{X}_k\subseteq X_k\subseteq\overline{X}_k$, where
$
\overline{X}_k:=\{x\in X: \left\langle \nabla d(x_k),x-x_k\right\rangle \ge 0\}.
$
\end{itemize}
Here $\nabla d(\cdot)$ denotes the gradient of $d(\cdot)$. Observe that the parameters $\alpha_k$ and $\beta$ are fixed a priori, and do not depend on any problem parameters. Moreover, the localizer $X_k$ is chosen between $\underline{X}_k$ and $\overline{X}_k$, so that the numbers of linear constraints in the two subproblems \eqref{APLsub1} and \eqref{APLsub2} can be fully controlled. It is shown in \cite{Lan13-1} that for
problem \eqref{eqnProblemX} with compact feasible set $X$, the APL method achieves the optimal iteration complexity
%${\cal O}\left(1/\epsilon^{\frac{2}{1+3\rho(X)}}\right)$
for any smooth, weakly smooth and non-smooth convex functions. Moreover, by  incorporating Nesterov's smoothing technique \cite{Nest05-1} into the APL method,
Lan also presented in \cite{Lan13-1}
the uniform smoothing level (USL) method which can achieve the optimal
complexity for solving an important class of nonsmooth structured saddle point (SP) problems
without requiring the input of any problem parameters
(see Subsection \ref{secFUSL} for more details).

\subsection{Contribution of this paper}
One crucial problem associated with most existing BL type methods, including APL and USL, is that each phase of these algorithms involves solving two optimization problems; first a linear programing problem to compute the lower bound, and then a constrained quadratic programing problem
to update the prox-center or new iterate. 
%they can only solve constrained problems with compact feasible set $X$.
In fact, the efficiency of these algorithms relies on the solutions of the two involved subproblems \eqref{APLsub1} and \eqref{APLsub2}, and the latter one
is often more complicated than the projection subproblem in the gradient projection type methods.
Moreover, most existing BL type methods require the assumption that the feasible set is
bounded due to the following two reasons. Firstly, the feasible set has to be bounded
to compute a meaningful lower bound by solving the 
aforementioned linear programing problem. 
Secondly, the convergence analysis of limited memory BL type methods (e.g., NERML, APL, and USL)
relies on the assumption that the feasible set is compact. 
These issues have significantly hindered the applicability of existing BL type methods.
Our contribution in this paper mainly consists of the following three aspects.

Firstly, we propose a novel bundle-level framework for unconstrained CP problems. The proposed framework solves unconstrained CP problems through solutions to a series of ball-constrained CPs. In particular, if there exists a uniformly optimal method (e.g., the APL and USL methods) that solves
ball-constrained black-box or structured CPs, then the proposed algorithm solves unconstrained black-box or structured CPs with optimal complexity
without requiring the input of any problem parameters as well. To the best of our knowledge, this is the first time in the literature that the complexity analysis has been performed for BL type methods to solve unconstrained CP problems (see Sections 3.2 and 3.3 in \cite{oliveira2014bundle} for more details).

Secondly, in order to solve ball-constrained CPs, we propose two greatly simplified BL type methods, namely the FAPL and FUSL methods, which achieve the same optimal iteration complexity as the APL and USL methods respectively, and maintain all the nice features of those methods, but has greatly reduced computational cost per iteration.
Such improvement has been obtained by the reduction and simplification of the subproblems that have to be solved in the APL and USL methods.
More specifically, we show that the linear optimization subproblem for computing the lower bound can be eliminated and that the ball constraint can be removed from the quadratic subproblem by properly choosing the prox-functions. We also generalize both FAPL and FUSL methods for solving strongly convex optimization problems and show that they can achieve the optimal iteration complexity bounds.

Thirdly,  we introduce a simple exact approach to solve the only subproblem in our algorithms. As mentioned earlier, the accuracy of the solutions to the subproblems is essential for the efficiency of all these BL type methods mentioned in Sections \ref{secBL} and \ref{secABL}.  By incorporating the proposed exact approach to solve the only subproblem in our algorithm, the accuracy of the FAPL and FUSL methods is significantly increased and the total number of the iterations required to compute an $\epsilon$-solution is greatly decreased comparing with the original APL and USL methods and other first order methods. Also when the number of linear constraints is fixed and small, the computational cost for solving the only subproblem only linearly depends on the dimension of the problem,
since the cost for computing the vector inner product will dominate that for solving a
few auxiliary linear systems. This feature is very important for solving 
large-scale CP problems.

Finally, we present very promising numerical results for these new FAPL and FUSL methods applied to solve large-scale
least square problems and total-variation based image reconstruction problems.
These algorithms significantly outperform other BL type methods, gradient type methods, and even the powerful MATLAB
solver for linear systems, especially when
the dimension and/or the Lipschitz constant of the problem is large.
%Moreover, if the optimal solution lies in a given ball centered with known radius (this prior information could be available in many practical problems), the iteration complexity and the computational cost for solving the only subproblem are dimension independent.

\vgap

It should be noted that there exist some variants of bundle-level methods \cite{BrKiKrLiPe95-1,DeSa14-1,BeDe13-1} for solving nonsmooth CP problems in which the computation of the subproblem to update the lower bound $\underline{f}_k$ is skipped, so that the feasible set $X$ is allowed to be unbounded. In each iteration, these methods apply a level feasibility check recursively in order to find a proper level $l_k$ and the associated level set $X_k$, and update $\underline{f}_k$ to $l_k$ if the level set associated with $l_k$ is empty. However, in these methods, repeatedly checking the emptiness of level sets associated with varied levels in each iteration may be very costly in practice, especially when the linear constraints in the level sets accumulate. Also, if the feasible set $X=\R^n$, then for any chosen level, the corresponding level set consisting of linear constraints is unlikely to be empty, which would result in the inefficiency for updating the lower bound. In \cite{BeDe13-1}, an alternative for updating the lower bound (or increasing the level) is introduced by comparing the distances from stability center to the newly generated iterate and the solution set. This approach requires some prior knowledge about the distance to the solution set, and a rough estimate for that may lead to incorrect lower bounds and improper choices of levels.

\subsection{Organization of the paper}
The paper is organized as follows. In Section \ref{secCCP}, we present a general scheme to extend the optimal BL type methods for unconstrained
convex optimization. In Section \ref{secF}, the new FAPL and FUSL methods are proposed followed by their convergence analysis, then an exact approach is introduced to solve the subproblem in these algorithms. We also extend the FAPL and FUSL methods to strongly convex CP and structured CP problems in Section \ref{secStronglyconvex}, following
some unpublished developments by Lan in \cite{Lan10-2}. The applications and promising numerical results are presented in Section \ref{secExperiments}.

\setcounter{equation}{0}
\section{Solving unconstrained CP problems through ball-constrained CP}
\label{secCCP}
Our goal in this section is to present a generic algorithmic framework to extend the uniformly optimal constrained BL algorithms in \cite{Lan13-1} for solving unconstrained problems.

Given $\overline{x}\in \R^n, R>0,\epsilon>0$, let us assume that there exists a first-order algorithm, denoted by  $\mathcal{A}(\overline{x},R,\epsilon)$, which can
find an $\epsilon$-solution of
\begin{align}
	\label{eqnCCP}
	f^*_{\overline{x}, R}:=\min_{x\in B(\overline{x},R)}f(x),
\end{align}
where $B(\overline{x},R)$ is defined in \eqref{def_Ball}.
In other words, we assume that each call to $\mathcal{A}(\overline{x}, R,\epsilon)$ will compute a point $z\in B(\overline{x},R)$ such that $f(z)-f^*_{\overline{x},R}\le \epsilon$.
Moreover, throughout this section, we assume that the iteration complexity associated with $\mathcal{A}(\overline{x}, R,\epsilon)$ is given by
\begin{align}
	\label{eqnCCPCompelxity}
	N_{\overline{x},R,\epsilon}:=\frac{C_1(\overline{x}, R, f)R^{\alpha_1}}{\epsilon^{\beta_1}}+\frac{C_2(\overline{x}, R, f)R^{\alpha_2}}{\epsilon^{\beta_2}},
\end{align}
where $\alpha_1\ge\beta_1>0,\alpha_2\ge\beta_2>0$ and $C_1(\overline{x}, R, f), C_2(\overline{x}, R, f)$ are some constants that depend on $f$ in \eqref{eqnCCP}. For example, if $f$ is a smooth convex function, $\nabla f$ is Lipschitz continuous in $\R^n$ with constant $L$, i.e.,  \eqref{smoothrelation} holds with $\rho(\R^n)=1$ and $M(\R^n)=L$, and we apply the APL method to \eqref{eqnCCP}, then we have only one term with $\alpha_1=1$, $\beta_1=1/2$, and $C_1(\overline{x}, R, f)=cL$ in \eqref{eqnCCPCompelxity}, where $c$ is a universal constant. 
Observe that the two complexity terms in \eqnok{eqnCCPCompelxity} will be useful for analyzing 
some structured CP problems in Section~\ref{secFUSL}. It should also be noted that a more accurate estimate of $C_1(\overline{x}, R, f)$ is $cM(B(\overline{x},R))$, since the Lipschitz constant $L=M(\R^n)$ throughout $\R^n$ is larger than or equal to the local Lipschitz constant in $B(\overline{x},R)$. 

By utilizing the aforementioned ball-constrained CP algorithm and a novel guess and check procedure, we present a bundle-level type algorithm for unconstrained convex optimiations as follows.

%in Algorithm \ref{algOracle} below, we show that the iteration complexity for solving unconstrained CP problems of \eqref{oriproblem} can be bounded by

\begin{algorithm}
\caption{\label{algOracle} Bundle-level type methods for unconstrained CP problems}
Choose initial estimation $r_0\le \|\overline{x}-x^{\ast}\|$ and compute the initial gap $\Delta_0:=f(\overline{x})-\min_{x\in B(\overline{x},r_0)}h(\overline{x},x)$.

For $k=0,1,2,\ldots$,
\begin{algorithmic}
\State  1. Set $\overline{x}_k'=\mathcal{A}(\overline{x}, r_k,\Delta_k)$ and $\overline{x}_k''=\mathcal{A}(\overline{x}, 2r_k,\Delta_k)$.
\State  2. If $f(\overline{x}_k')-f(\overline{x}_k'')> \Delta_k$, update $r_k\leftarrow 2r_k$ and go to step 1\label{algOraclestep2}.
\State  3. Otherwise, let $\overline{x}_k^{\ast}=\overline{x}_k''$, $\Delta_{k+1}={\Delta_k}/{2}$ and $r_{k+1}=r_k$.

\end{algorithmic}
\end{algorithm}
Step 1 and Step 2 in Algorithm \ref{algOracle} constitute a loop to find a pair of solution $(\overline{x}_k',\overline{x}_k'')$ satisfying $0\le f(\overline{x}_k')-f(\overline{x}_k'')\le \Delta_k$. Since $\overline{x}_k'$ and $\overline{x}_k''$ are $\Delta_k$-optimal solutions to $\min_{x\in B(\overline{x},r_k)}f(x)$ and $\min_{x\in B(\overline{x},2r_k)}f(x)$, respectively, this loop must terminate in finite time, because it will terminate whenever $r_k\ge D^*$, where $D^{\ast}$ is defined in \eqref{eqnxstar}.
For simplicity, we call it an expansion if we double the radius in Step 2, and each iteration may contain several expansions before updating the output solution $\overline{x}_k^{\ast}$ in step 3.

Before analyzing the rate of convergence for Algorithm~\ref{algOracle}, we discuss some important observations related to the aforementioned expansions.
\begin{lemma}
\label{unconstrainlemma}
Suppose that $\overline{x}$ is a fixed point and $R>0$ is a fixed constant. Let $\overline{x}_1$ and $\overline{x}_2$ be $\epsilon$-solutions to problems
\begin{equation}
\label{subconstLemma}
f_1^{\ast} :=\min_{x\in B(\overline{x}, R)}f(x)  \ \ \ \mbox{and} \ \ \ f_2^{\ast} :=\min_{x\in B(\overline{x}, 2R)}f(x),
\end{equation}
%\begin{align}
%f_1^{\ast}&:=\min_{x\in B(x_0, R)}f(x)\label{probconsR}
%\end{align}
%and
%\begin{align}
%f_2^{\ast}&:=\min_{x\in B(x_0, 2R)}f(x)\label{procons2R}
%\end{align}
respectively.
If $0\le f(\overline{x}_1)-f(\overline{x}_2)\le \epsilon$, then we have
\begin{equation}
	\label{eqnUnconstraintLemma}
  f(\overline{x}_2)-f^{\ast}\le \left(3+\frac{2D^{\ast}}{R} \right)\epsilon,
\end{equation}
where $f^*$ and $D^{\ast}$ are defined in \eqref{oriproblem} and \eqref{eqnxstar} respectively.
\end{lemma}

\begin{proof}
Clearly, by definition, we have
$\|\overline{x}_1-\overline{x}\|\le R$, $\|\overline{x}_2-\overline{x}\|\le 2R$, $0\le f(\overline{x}_1)-f_1^{\ast}\le \epsilon$,
and
$0\le f(\overline{x}_2)-f_2^{\ast}\le \epsilon.$
It suffices to consider the case when $f_2^{\ast}> f^{\ast}$ and $\|x^*-\overline{x}\|>2R$, since otherwise \eqref{eqnUnconstraintLemma} holds trivially.
Suppose $x_1^{\ast}$ and $x_2^{\ast}$ are the solutions to the first and second problems in \eqref{subconstLemma} respectively,  let $\hat{x}$ be the intersection of the line segment $(x^{\ast},x_1^{\ast})$ with the ball $B(\overline{x},2R)$, and denote $R_1:=\|\hat{x}-x_1^{\ast}\|$ and $R_2:=\|x^{\ast}-x_1^{\ast}\|$. Clearly, $\hat{x}=(1-\frac{R_1}{R_2})x_1^{\ast}+\frac{R_1}{R_2}x^{\ast}$. By the convexity of $f(\cdot)$, we have
\begin{equation}\label{reluncons}
  f(\hat{x})\le (1-\frac{R_1}{R_2})f(x_1^{\ast})+\frac{R_1}{R_2}f(x^{\ast}),
\end{equation}
which implies that
\begin{equation}\label{relunconsconvex}
\frac{R_1}{R_2}[f(x_1^{\ast})-f(x^{\ast})]\le f(x_1^{\ast})-f(\hat{x}),
\end{equation}
and that $f(\hat{x})\le f(x_1^{\ast})$ due to the fact that $f(x^{\ast})\le f(x_1^{\ast})$.
Also, we have $f(\hat{x})\ge f(x_2^{\ast})$ since $\hat{x} \in B(\overline{x}, 2R)$. In addition,
\begin{align}
  f(x_1^{\ast})-f(x_2^{\ast})&=[f(x_1^{\ast})-f(\overline{x}_1)]+[f(\overline{x}_1)-f(\overline{x}_2)]+[f(\overline{x}_2)-f(x_2^{\ast})]\\
                             &\le 0+\epsilon+\epsilon=2\epsilon.
\end{align}
Combining the previous inequalities, we obtain
\begin{equation}
  f(x_1^{\ast})-2\epsilon \le f(x_2^{\ast})\le f(\hat{x})\le f(x_1^{\ast}),
\end{equation}
which implies that $f(x_1^{\ast})-f(\hat{x})\le 2\epsilon$. Using the previous conclusion \eqref{relunconsconvex}, and the fact that $R_1\ge R$ and $R_2\le D^{\ast}+R$, we have
$$f(x_1^{\ast})-f(x^{\ast})\le \frac{2\epsilon R_2}{R_1}\le\left(2+\frac{2D^{\ast}}{R}\right)\epsilon. $$\
Therefore,
$$f(\overline{x}_2)-f(x^{\ast})\le f(\overline{x}_1)-f(x^{\ast})\le [f(\overline{x}_1)-f(x_1^{\ast})]+[f(x_1^{\ast})-f(x^{\ast})]\le \left(3+\frac{2D^{\ast}}{R} \right)\epsilon.$$
\end{proof}

\vgap
We are now ready to prove the iteration complexity of Algorithm \ref{algOracle} for solving the unconstrained CP problem in \eqref{oriproblem}.

\begin{theorem}
\label{unconstraintheorem}
Suppose that the number of evaluations of $f'$ in one call to $\mathcal{A}(\overline{x}, R,\epsilon)$ is bounded by \eqref{eqnCCPCompelxity},
and denote $\epsilon_k:=f(\overline{x}_k^{\ast})-f^{\ast}$
for the iterates $\{\overline{x}_k^{\ast}\}$ of Algorithm \ref{algOracle}.
Then we have
\begin{enumerate}[a)]
	\item $r_k< 2D^{\ast}$ for all $k$;
	\item $\lim_{k\to \infty}\epsilon_k=0$;
	\item The total number of evaluations of $f'$ performed by Algorithm \ref{algOracle}
	up to the $k$-th iteration
	is bounded by
\begin{equation}
\label{iterboundofunconst}
	\cO\left(\frac{C_1(\overline{x}, 4D^{\ast}, f) (D^*)^{\alpha_1}}{\epsilon_k^{\beta_1}}+\frac{C_2(\overline{x}, 4D^{\ast}, f) (D^*)^{\alpha_2}}{\epsilon_k^{\beta_2}}\right),
\end{equation}
where $D^{\ast}$ is defined in \eqref{eqnxstar}.
%\begin{equation}
%  N:=\frac{(1+2^{\alpha})2^{\alpha}5^{\beta}C_{x_0,4\|x^{\ast}-x_0\|,f}}{2^{\beta}-1}\cdot
%    \frac{ \|x^{\ast}-x_0\|^{\alpha}}{\epsilon_k^{\beta}}.
%\end{equation}
\end{enumerate}
\end{theorem}
\begin{proof}
We start by proving that $r_k<  2D^{\ast}$ for all $k$. From the description of Algorithm \ref{algOracle}, we see that expansions occur if and only if $f(\overline{x}_k')-f(\overline{x}_k'')>\Delta_k$ at Step \ref{algOraclestep2}. Moreover, we can observe that $f(\overline{x}_k')-f(\overline{x}_k'')\le \Delta_k$ if $x^{\ast}\in B(\overline{x},r_k)$. This observation implies that $r_k<  2D^{\ast}$. Indeed, it is easy to see that the total number of expansions is bounded by
\begin{equation}
  \overline{S}_1:=\left\lceil\log_2{\frac{D^{\ast}}{r_0}}\right\rceil+1.
\end{equation}

To prove b), noting from the description of Step 3 and Lemma \ref{unconstrainlemma} that
\begin{equation}
\label{unconstrainedrelation}
  \epsilon_k=f(\overline{x}_k'')-f^{\ast}\le \left(3+\frac{2D^{\ast}}{r_k} \right)\Delta_k.
\end{equation}
Since the total number of expansions is bounded by $\overline{S}_1$, and $\Delta_k$ decreases to $0$
as $k$ increases, we have $\lim_{k\to \infty}\epsilon_k= 0$.

To prove c), assume that the number of executions of Step 1 performed by $\mathcal{A}$ to find $\overline{x}_k^{\ast}$ is $K$. Now we can estimate the total number of evaluations of $f'$ performed by the $K$ executions. For any $0\le j<K$, as $2^{\alpha_1}\ge2^{\beta_1}>1$ and $2^{\alpha_2}\ge2^{\beta_2}>1$, we have
\begin{equation}
N_{j+1}'\ge2^{\beta_1}N_j'
\text{\ and\ } N_{j+1}''\ge 2^{\beta_2}N_j''
\end{equation}
by the assumption of $N_{\overline{x},R,\epsilon}$, where $N_j=N_j'+N_j''$ denotes the bound on iteration number for the $j^{th}$ execution, and $N_j',N_j''$ correspond with the first and second terms in \eqref{iterboundofunconst} respectively.  Then the total number of iterations is bounded by
\begin{align}
  N:&=\sum_{j=1}^K (N_j'+N_j'') \le N_K'\sum_{j=0}^{K-1} (2^{\beta_1})^{-j}+N_K''\sum_{j=0}^{K-1} (2^{\beta_2})^{-j}\\
    &<  N_K'\sum_{j=0}^{+\infty} 2^{-\beta_1 j}+ N_K''\sum_{j=0}^{+\infty} 2^{-\beta_2 j}\le \frac{1}{1-2^{-\beta_1}}N_K'+\frac{1}{1-2^{-\beta_2}}N_K''\\
    &\le \frac{(1+2^{\alpha_1})C_1(\overline{x}, 2r_k, f)}{1-2^{-\beta_1} }\cdot \frac{r_k^{\alpha_1}}{\Delta_k^{\beta_1}}+ \frac{(1+2^{\alpha_2})C_2(\overline{x}, 2r_k, f)}{1-2^{-\beta_2} }\cdot \frac{r_k^{\alpha_2}}{\Delta_k^{\beta_2}}.
\end{align}

Combining the above inequality with \eqref{unconstrainedrelation}, we have
\begin{equation}
N<\sum_{i=1}^2\frac{(1+2^{\alpha_i})C_i(\overline{x}, 2r_k, f)}{1-2^{-\beta_i}}\cdot
    \frac{r_k^{\alpha_i}(3+\frac{2D^{\ast}}{r_k})^{\beta_i}}{\epsilon_k^{\beta_i}}.
\end{equation}
Since $\alpha_i\ge\beta_i>0$, then $r_k^{\alpha_i}(3+\frac{2D^{\ast}}{r_K})^{\beta_i}=r_k^{\alpha_i-\beta_i}(3r_k+2D^{\ast})^{\beta_i}$ for $i=1,2$ is monotonically increasing with respect to $r_k$, which, in view of the fact $r_k<  2D^{\ast}$ for any $k\ge 0$, then clearly implies

\begin{equation}
  N< \sum_{i=1}^2 \frac{(1+2^{\alpha_i})2^{\alpha_i+3\beta_i}C_i(\overline{x},4D^{\ast},f)}{2^{\beta_i}-1}\cdot
    \frac{ (D^{\ast})^{\alpha_i}}{\epsilon_k^{\beta_i}}.
\end{equation}
Hence the proof is complete.
\end{proof}

Note that to solve the unconstrained black-box CP problem \eqref{oriproblem}, the termination criterions of most first-order algorithms are based on the residual of the gradient or gradient mapping, which would lead to different complexity analysis. To the best of our knowledge, without any prior information on $D^{\ast}$, there is no any termination criterion based on functional optimality gap that could guarantee the termination of algorithms for finding an $\epsilon$-solution of \eqref{oriproblem}. Comparing to Nesterov's optimal gradient method for unconstrained problems in \cite{Nest83-1}, Algorithm \ref{algOracle} only provides efficiency estimates about $\epsilon_k:=f(\overline{x}^{\ast}_k)-f^{\ast}$ when the output $\overline{x}^{\ast}_k$ is updated, while the optimal gradient method could have estimates about $\overline{\epsilon}_k:=f(x_k)-f^{\ast}$ for each iterate $x_k$. For both methods the efficiency estimates involve $D^{\ast}$.  Since Algorithm \ref{algOracle} extend methods for ball-constraint CP problems to solve \eqref{oriproblem}, and the iterations in the expansions of Algorithm \ref{algOracle} could be regarded as a guess and check procedure to determine $D^{\ast}$, it is reasonable that the efficiency estimates are only provided for unexpansive steps which update $\overline{x}^{\ast}_k$.

Since in \cite{Lan13-1} the APL method, and its variant the USL method, have optimal iteration complexity for solving smooth, nonsmooth, weakly smooth CP problems and structured nonsmooth CP problems on compact feasible sets, Algorithm \ref{algOracle} could be incorporated to solve \eqref{oriproblem} by specifying feasible sets to a sequences of Euclidean balls. Therefore, the main problem remained is how to improve the efficiency of these BL type methods for solving ball-constrained CP problems.

\setcounter{equation}{0}
\section{Fast prox-level type methods for ball-constrained and unconstrained problems} \label{secF}

This section contains four subsections. We first present a much simplified APL method, referred to the
fast APL (FAPL) method, for solving ball-constrained black-box CP problems in Subsection~\ref{secFAPL}, and
then present the fast USL (FUSL) method for solving a special class of ball-constrained structured CP problems
in Subsection~\ref{secFUSL}. We
show how to solve the subproblems in these two algorithms in Subsection~\ref{secKKT}.
We also briefly discuss in Subsection~\ref{secunconstraint} the applications of the FAPL and FUSL methods for unconstrained
optimization, based on our results in Section~\ref{secCCP}. For the sake of simplicity, throughout this section, we denote $\rho=\rho(B(\overline{x},R)), M=M(B(\overline{x},R))$ for \eqref{eqnCCP}.

%In this section we present our novel fast bundle level methods for solving  \eqref{oriproblem} with the assumption that there exists an optimal solution lies in a given ball with the radius R to be known. By using the geometry of the ball the novel algorithms, FAPL and FUSL,  simplified APL and USL schemes, greatly reduce computational cost, and improve practical performance, while they still maintain the optimal rate of convergence as APL and USL. If such prior information for the assumption aforementioned is not available, combined with the extension algorithm \ref{algOracle} described in the previous section,  the FAPL and FUSL play a crucial role in efficiently solving the unconstrained problem \eqref{oriproblem} (i.e. $X=R^n$).

%In the previous section, we assume that there exists an algorithm for solving ball-constrained CP \eqref{eqnCCP}, and we call such algorithm through $Oracle(x_0,R,\epsilon)$. The choice of such algorithm is critical to the performance of unconstrained CP. As \eqref{eqnCCP} is a CP problem with compact feasible set, there exists many first-order algorithms that achieves optimal iteration complexity, e.g., the APL and USL method. In this section, we demonstrate that the structure of ball-constraint in \eqref{eqnCCP} is useful for more efficient optimal method. In particular, we propose modifications of APL and USL method that achieves the same iteration complexity as APL and USL, but with reduced computational cost for solving ball-constrained CP problems.

\subsection{FAPL for ball-constrained black-box problems}
\label{secFAPL}
Our goal in this subsection is to present the FAPL method, which can significantly reduce the iteration
cost  for the APL method applied to problem \eqnok{eqnCCP}.
In particular, we show that only one subproblem, rather than two subproblems (see \eqref{APLsub1} and \eqref{APLsub2})
as in the APL method, is required in the FAPL method for defining a new iterate (or prox-center) and updating lower bound. We also demonstrate that the ball constraint in \eqnok{eqnCCP} can be eliminated from the subproblem by properly specifying the
prox-function.

Similarly to the APL method, the FAPL method consists of multiple phases, and in each phase the FAPL gap reduction procedure,
denoted by $\mathcal{G}_{FAPL}$, is called to reduce the gap between the upper and lower bounds on $f^\ast_{\overline{x}, R}$ in \eqref{eqnCCP}  by a constant factor.

We start by describing the FAPL gap reduction procedure in Procedure \ref{funFAPL}. This procedure differs from the
gap reduction procedure used in the APL method in the following several aspects.
Firstly, the feasible sets $\underline{Q}_k$ and $\overline{Q}_k$ (see steps \ref{stepGFAPLcut} and \ref{stepGFAPLQ}) in procedure $\mathcal{G}_{FAPL}$
only contain linear constraints
and hence are possibly unbounded, while the localizers in the APL method must be compact.
Secondly, we eliminate the subproblem that updates the lower bound on $f^*$ in the APL method. Instead, in the FAPL method, the lower bound
is updated to $l$ directly whenever $\underline{Q}_k = \emptyset$ or $\|x_k-\overline{x}\|>R$.
Thirdly, we choose a specific prox-function $d(x)=\frac{1}{2}\|x-\overline{x}\|^2$, and the prox-center is fixed to be $\overline{x}$.
As a result, all the three sequences $\{x_k\},\{x_k^l\}$ and $\{x_k^u\}$ will reside in the ball $B(\overline{x},R)$.
%It should be noted that such ball restrictions play a similar role to the compactness of the convex feasible set in the APL method.
At last, as we will show in next subsection, since the subproblem \eqref{FAPLsub} only contains a limited number of linear constraints (depth of memory),
we can solve it very efficiently, or even exactly if the depth of memory is small, say, $\le 10$.
%this subproblem can solved exactly.
%we can use a non-iterative approach to solve \eqref{FAPLsub} exactly with almost dimension free computational cost.

\begin{procedure}%[h]
\caption{\label{funFAPL}The FAPL gap reduction procedure: $(x^+,
\lb^+)=\mathcal{G}_{FAPL}(\hat{x}, \lb, R,\overline{x},\beta,\theta)$}
\begin{algorithmic}[1]
		\makeatletter
		\setcounter{ALG@line}{-1}
		\makeatother
\State \label{stepGFAPLinit}
Set $k=1$, $\overline{f}_0=f(\hat{x}), l=\beta \cdot \lb+(1-\beta)\overline{f}_0$, $Q_0=\R^n$, and $x_0^u=\hat{x}$.
Let $x_0 \in B(\overline x, R)$ be given arbitrarily.
% and choose the prox-function $d(x)=\frac{1}{2}\|x-\overline{x}\|^2$.

\State \label{stepGFAPLcut}
Update the cutting plane model:
\begin{align}
x_k^l&=(1-\alpha_k)x_{k-1}^u+\alpha_k x_{k-1},\label{def_xkl} \\
h(x_k^l,x)&=f(x_k^l)+\left\langle f'(x_k^l),x-x_k^l \right\rangle , \label{def_hk}\\
\underline{Q}_k&=\{x\in Q_{k-1}:h(x_k^l,x)\le l \}\label{DEFQ_kLow}.
\end{align}
\State \label{stepGFAPLp}
Update the prox-center and lower bound:
\begin{equation}
\label{FAPLsub}
x_k=\argmin_{x\in \underline{Q}_k}\left\{ d(x):= \frac{1}{2}\|x-\overline{x}\|^2\right\}.
\end{equation}
If $\underline{Q}_k=\emptyset$ or $\|x_k-\overline{x}\|> R$, then \textbf{terminate} with outputs $x^+=x_{k-1}^u,\lb^+=l $.
\State \label{stepGFAPLub}
Update the upper bound: set
\begin{align}
\tilde{x}_k^u&=(1-\alpha_k)x_{k-1}^u + \alpha_k x_k,\label{eqnxku}\\
	x_k^u&=\begin{cases}
		\tilde{x}_k^u, &\text{if\ } f(\tilde{x}_k^u)< \overline{f}_k, \\
		 x_{k-1}^u, & \text{otherwise},
		  \end{cases} \label{eqnxku1}
\end{align}
and $\overline{f}_k=f(x_k^u)$. If $\overline{f}_k\le l+\theta(\overline{f}_0-l)$, then \textbf{terminate} with $x^+=x_{k}^u,\lb^+=\lb$.
\State \label{stepGFAPLQ}
Choose any polyhedral set $Q_k$ satisfying $\underline{Q}_k\subseteq Q_k\subseteq \overline{Q}_k$, where
\begin{equation}
\overline{Q}_k:=\{x\in \mathbb{R}^n:\left\langle x_k-\overline{x}, x-x_k\right\rangle \ge 0\}.\label{DEFQ_KUp}
\end{equation}
\ \ \ Set $k=k+1$ and go to step \ref{stepGFAPLcut}.
\end{algorithmic}
\end{procedure}

We now add a few more remarks about the technical details of Procedure \ref{funFAPL}. Firstly, Procedure \ref{funFAPL} is
terminated at step 2 if $Q_k=\emptyset$ or $\|x_k-\overline{x}\|> R$, which can be checked automatically when solving the
subproblem \eqref{FAPLsub} (see Subsection \ref{secKKT} for more details).
Secondly, in step \ref{stepGFAPLQ}, while $Q_k$ can be any polyhedral set between $\underline{Q}_k$ and $\overline{Q}_k$, in practice we can simply choose $Q_k$ to be the intersection of the half-space $\{x\in \R^n: \left \langle x_k-\overline{x},x-x_k\right \rangle \ge 0\}$ and a few most recently generated half-spaces, each of which is defined by $\{x\in \mathbb{R}^n:h(x_\tau^l,x)\le l\}$ for some $1\le \tau\le k$.
Finally, in order to guarantee the termination of procedure $\mathcal{G}_{FAPL}$ and the optimal iteration complexity, the parameters $\{\alpha_k\}$ used in this procedure need to be properly chosen. One set of conditions that $\{\alpha_k\}$ should satisfy
to guarantee the convergence of procedure $\mathcal{G}_{FAPL}$  is readily given in \cite{Lan13-1}:
\begin{equation}
\label{stepsize}
\alpha_1=1,\ 0<\alpha_k\le 1,\ \gamma_k\|\tau_k(\rho)\|_{\frac{2}{1-\rho}}\le ck^{-\frac{1+3\rho}{2}},\ \forall k\ge 1
\end{equation}
for some constants $c>0$, where $\|\cdot\|_p$ denotes the $l_p$ norm,
\begin{equation}
\gamma_k:=
 \begin{cases}
 1, & k=1, \\
 \gamma_{k-1}(1-\alpha_k), & k\ge 2,
 \end{cases}\ \ \text{and\ \ }    \tau_k(\rho):=\left\{\frac{\alpha_1^{1+\rho}}{\gamma_1},\frac{\alpha_2^{1+\rho}}{\gamma_2},\ldots,\frac{\alpha_k^{1+\rho}}{\gamma_k} \right\}.
\end{equation}

The following lemma, whose proof is given in \cite{Lan13-1}, provides two examples for the selection of $\{\alpha_k\}$.
\begin{lemma}
\begin{itemize}
\item [a)] If
$\alpha_k={2}/(k+1)$,
$k=1,2,\ldots$, then the condition \eqref{stepsize} is satisfied with $c=2^{1+\rho}3^{-\frac{1-\rho}{2}}$.
\item [b)]  If $\{\alpha_k\}$ is recursively defined by
\begin{equation}
\alpha_1=\gamma_1=1, \alpha_k^2=(1-\alpha_k)\gamma_{k-1}=\gamma_k, \forall k\ge 2,
\end{equation}
 then the condition \eqref{stepsize} holds with $c=4/3^{\frac{1-\rho}{2}}$.
\end{itemize}
\end{lemma}

The following lemma describes some important observations regarding the execution of procedure $\mathcal{G}_{FAPL}$.
\begin{lemma}
\label{lemTech1}
Let $\mathcal{E}_f(l):=\{x\in B(\overline{x},R):f(x)\le l\}$. If $\cE_f(l)\neq \emptyset$, then the following statements hold for
procedure $\mathcal{G}_{FAPL}$.
\begin{enumerate}[a)]
\item Step \ref{stepGFAPLQ} is always well-defined unless procedure $\mathcal{G}_{FAPL}$ already terminated\label{lemTech1partA}.
\item $\cE_f(l)\subseteq \underline{Q}_k\subseteq Q_k \subseteq \overline{Q}_k$ for any $k\ge 1$\label{lemTech1partB}.
\item  If $\underline{Q}_k \neq \emptyset$, then problem \eqref{FAPLsub} in step \ref{stepGFAPLp} has a unique solution. Moreover, if
procedure $\mathcal{G}_{FAPL}$ terminates at step \ref{stepGFAPLp}, then $l\le f^{\ast}$\label{lemTech1partC}.
\end{enumerate}
\end{lemma}

\begin{proof}
To prove part \ref{lemTech1partA}, we will use induction to prove that $\cE_f(l)\subseteq Q_k$ for all $k\ge 0$. Firstly, as $Q_0$ is set to $\R^n$, we have $\cE_f(l)\subseteq Q_0$. Moreover, if $\cE_f(l)\subseteq Q_{k-1}$ for some $k\ge 1$, then from the definition of $\underline{Q}_k$ in \eqref{DEFQ_kLow} and the observation that $h(x_k^l,x)\le f(x)\le l$ for all $x\in\cE_f(l)$, we have $\cE_f(l)\subseteq \underline{Q}_k\subseteq Q_k$, hence part a) holds.

To prove \ref{lemTech1partB}, it suffices to show that $\underline{Q}_k\subseteq \overline{Q}_k$, since $Q_k$ is chosen between $\underline{Q}_k$ and $\overline{Q}_k$, and $\cE_f(l)\subseteq \underline{Q}_k$ is proved from the above induction. By the definition of $\overline{Q}_k$ in \eqref{DEFQ_KUp}, we have $\overline{Q}_k = \{x\in\R^n:d(x)\ge d(x_k) \}$, hence $\underline{Q}_k\subseteq \overline{Q}_k$, and part \ref{lemTech1partB} holds.

We now provide the proof of part \ref{lemTech1partC}. From the definition of $Q_k$ in step \ref{stepGFAPLQ} and the definition of $\underline{Q}_k$ in \eqref{DEFQ_kLow} we can see that $\underline{Q}_k$ is the intersection of half-spaces, hence it is convex and closed. Therefore, the subproblem \eqref{FAPLsub} always has a unique solution as long as $\underline{Q}_k$ is non-empty.

To finish the proof it suffices to show that $\cE_f(l)=\emptyset$ when either $\underline{Q}_k=\emptyset$ or $\|x_k-\overline{x}\|>R$, which can be proved by contradiction. Firstly, if $\underline{Q}_k=\emptyset$ but $\cE_f(l)\not=\emptyset$, then by part \ref{lemTech1partB} proved above, we have $\cE_f(l)\subseteq \underline{Q}_k$, which contradicts the assumption that $\underline{Q}_k$ is empty. On the other hand, supposing that $\|x_k-\overline{x}\|> R$ and $\cE_f(l)\not=\emptyset$, let $x_R^{\ast}:=\argmin_{x\in B(\overline{x},r)}f(x)$, it is clear that $x_R^*\in\cE_f(l)\subseteq \underline{Q}_k$ by \ref{lemTech1partB}, however $\|x_R^{\ast}-\overline{x}\|\le R <\|x_k-\overline{x}\|$ which contradicts the definition of $x_k$ in \eqref{FAPLsub}.
\end{proof}

The following lemma shows that whenever procedure $\mathcal{G}_{FAPL}$ terminates, the
gap between the upper and lower bounds on $f^*_{\overline{x},R}$
is reduced by a constant factor.

\begin{lemma}
\label{lemq}
Let $\ub:=f(\hat{x}),\ub^+:=f(x^+)$ in procedure $\mathcal{G}_{FAPL}$.
Whenever procedure $\mathcal{G}_{FAPL}$ terminates, we have
$\ub^+-\lb^+\le q(\ub-\lb)$, where
\begin{equation}
\label{qdef}
q:=\max\{\beta, 1-(1-\theta)\beta\}.
\end{equation}
\end{lemma}

\begin{proof}
By the definition of $x_k^u$ in \eqref{eqnxku} and the definition of $\overline{f}_k$ in step \ref{stepGFAPLub} of procedure
$\mathcal{G}_{FAPL}$, we have
$\overline{f}_{k}\le \overline{f}_{k-1},\ \forall k\ge 1$, which implies $\ub^+\le \ub$.
Procedure $\mathcal{G}_{FAPL}$ could terminate at either step \ref{stepGFAPLp} or \ref{stepGFAPLub}. We first suppose that it terminates at step \ref{stepGFAPLp} after $k$ iterations. Using the termination condition $\lb^+=l=\beta \cdot \lb+(1-\beta)\ub$, we have
$$ \ub^+-\lb^+\le \ub -\beta \, \lb+(1-\beta)\ub= \beta(\ub -\lb).$$
Now suppose that Procedure \ref{funFAPL} terminates at step \ref{stepGFAPLub} after $k$ iterations. We have $\ub^+=\overline{f}_k\le l+\theta(\ub-l)$ and $\lb^+\ge \lb$. Using the fact that $l=\beta\cdot \lb+(1-\beta)\ub$, we conclude that
$$\ub^+-\lb^+\le l+\theta \, (\ub-l)-\lb=[1-(1-\theta)\beta](\ub-\lb). $$
We conclude the lemma by combining the above two relations.
\end{proof}

We now provide a bound on the number of iterations performed by procedure  $\mathcal{G}_{FAPL}$.
Note that the proof of this result is similar to Theorem 3 in \cite{Lan13-1}.
\begin{proposition} \label{prop_FAPL_gap}
If the stepsizes $\{\alpha_k\}_{k\ge 1}$ are chosen such that \eqref{stepsize} holds,
then the number of iterations performed by procedure $\mathcal{G}_{FAPL}$ does not exceed
\begin{equation}
\label{numofeachgapred}
N(\Delta):= \left(\frac{cMR^{1+\rho}}{(1+\rho)\theta\beta\Delta} \right)^{\frac{2}{1+3\rho}}+1,
\end{equation}
where $\Delta:=\ub-\lb$.
\end{proposition}

\begin{proof}
It can be easily seen from the definition of $\overline{Q}_k$ that $x_k=\argmin_{x\in \overline{Q}_k}d(x)$,
which, in view of the fact that $Q_k\subseteq \overline{Q}_k$, then implies that $x_k= \argmin_{x\in Q_k} d(x)$.
Using this observation, and the fact that $x_{k+1}\in Q_k$ due to \eqref{DEFQ_kLow} and \eqref{FAPLsub},
we have
$\left\langle  \nabla d(x_k),x_{k+1}-x_k\right\rangle \ge 0.$
Since $d(x)$ is strongly convex with modulus 1, we have
$$d(x_{k+1})\ge d(x_k)+\left\langle  \nabla d(x_k),x_{k+1}-x_k\right\rangle +\frac{1}{2}\|x_{k+1}-x_k\|^2.$$
Combining the above two relations, we conclude $\frac{1}{2}\|x_{k+1}-x_k\|^2\le d(x_{k+1})-d(x_k)$.
Summing up these inequalities
for any $k \ge 1$, we conclude that
 \begin{equation}
 \label{Bounddk}
 \frac{1}{2}\sum\limits_{\tau=1}^{k}\|x_{\tau+1}-x_\tau \|\le d(x_{k+1})=\frac{1}{2}\|x_{k+1}-\overline{x}\|^2\le \frac{1}{2}R^2.
 \end{equation}
Now suppose that procedure $\mathcal{G}_{FAPL}$ does not terminate at the $k^{th}$ iteration.
Applying the relation (3.14) in \cite{Lan13-1}, and notice that $\alpha_1 = 1$, we have
\begin{equation} \label{pfrelation}
f(x_k^u)-l\le\frac{M}{1+\rho}[2d(x_k)]^{\frac{1+\rho}{2}}\gamma_k\|\tau_k(\rho)\|_{\frac{2}{1-\rho}}.
\end{equation}
In view of steps \ref{stepGFAPLp}
and \ref{stepGFAPLub} in procedure \ref{funFAPL}, and using the fact that $l=\beta\cdot \lb + (1-\beta)\ub$ in step \ref{stepGFAPLinit}, we have
$$f(x_k^u)-l>\theta(\ub-l)=\theta\beta\Delta.$$
Combining the above two relations, and using \eqref{stepsize} and \eqref{Bounddk}, we obtain
\begin{equation}
  \theta\beta\Delta< \frac{MR^{1+\rho}}{(1+\rho)}\cdot \frac{c}{k^{\frac{1+3\rho}{2}}},
\end{equation}
which implies that
\begin{equation}
  k<\left(\frac{cMR^{1+\rho}}{(1+\rho)\theta\beta\Delta} \right)^{\frac{2}{1+3\rho}}.
\end{equation}
\end{proof}

\vgap

In view of Lemma~\ref{lemq} and Proposition~\ref{prop_FAPL_gap},
we are now ready to describe the FAPL method, which performs a sequence of calls to procedure $\mathcal{G}_{FAPL}$
until an approximate solution with sufficient accuracy is found.
\begin{algorithm}[h]
\caption{\label{algFAPL}The fast accelerated prox-level (FAPL) method}
\begin{algorithmic}[1]
		\makeatletter
		\setcounter{ALG@line}{-1}
		\makeatother
\State \label{stepFAPLinit}Given ball $B(\overline{x},R)$, choose initial point $p_0\in B(\overline{x},R)$, tolerance $\epsilon>0$ and parameters $\beta,\theta \in(0,1)$.
\State \label{stepFAPLinitbound}Set $p_1\in \Argmin_{x\in B(\overline{x},R)} h(p_0,x)$, $\lb_1=h(p_0,p_1), \ub_1=\min\{f(p_0),f(p_1)\}$, let $\hat{x}_1$ be either $p_0$ or $p_1$ such that $f(\hat{x}_1)=\ub_1$,  and $s=1$.
\State \label{stepFAPLmain}If $\ub_s-\lb_s\leq \epsilon$, \textbf{terminate} and \textbf{output} approximate solution $\hat{x}_s$.
\State \label{stepFAPLcall}Set $(\hat{x}_{s+1}, \lb_{s+1})=\mathcal{G}_{FAPL}(\hat{x}_{s}, \lb_s, R, \overline{x},\beta,\theta)$ and  $\ub_{s+1}=f(\hat{x}_{s+1})$.
\State Set $s=s+1$ and go to step \ref{stepFAPLmain}.
\end{algorithmic}
\end{algorithm}

A phase of the FAPL method occurs whenever $s$ increments by $1$. For the sake of simplicity, each iteration of
procedure $\mathcal{G}_{FAPL}$ is also referred to as an iteration of the FAPL method. The following theorem establishes
the complexity bounds on
the total numbers of phases and iterations performed by the FAPL method and its proof is similar to that of
Theorem 4 in \cite{Lan13-1}.

\begin{theorem}
\label{thmFAPL}
If the stepsizes $\{\alpha_k\}$ in procedure $\mathcal{G}_{FAPL}$ are chosen such that \eqref{stepsize} holds, then the following statements
hold for the FAPL method.
\begin{itemize}
\item [a)] The number of phases performed by the FAPL method does not exceed
\begin{equation}
\label{numofFAPLgapred}
S:=\left\lceil \max\left\{0,\log_{\frac{1}{q}} \left(\frac{(2R)^{1+\rho}M}{(1+\rho)\epsilon} \right) \right\}\right\rceil.
\end{equation}
\item [b)] The total number of iterations performed by the FAPL method for computing an $\epsilon$-solution of problem \eqref{eqnCCP} can be bounded by
\begin{equation}
\label{eqnFAPL}
N(\epsilon):=  S+\frac{1}{1-q^{\frac{2}{1+3\rho}}}\left(\frac{cMR^{1+\rho}}{(1+\rho)\theta\beta\epsilon}  \right)^{\frac{2}{1+3\rho}},
\end{equation}
where $q$ is defined in \eqref{qdef}.
%and $\tilde{S}$ is defined in \eqref{numofFAPLgapred}.
\end{itemize}
\end{theorem}

\begin{proof}
We first prove part a). Let $\Delta_s:=\ub_s-\lb_s$, without loss of generality, we assume that $\Delta_1>\epsilon$. In view of step \ref{stepFAPLinit} in the FAPL method and \eqref{smoothrelation}, we have
 \begin{equation}
 \label{initialDelta}
 \Delta_1\le f(p_1)-h(p_0,p_1)=f(p_1)-f(p_0)-\left\langle f'(p_0),p_1-p_0\right\rangle \le \frac{(2R)^{1+\rho}M}{1+\rho}.
 \end{equation}
Also, by Lemma \ref{lemq} we can see that $\Delta_{s+1}\le q\Delta_s$ for any $s\ge 1$, which implies that
$$\Delta_{s+1}\le q^s \Delta_1,\ \forall s\ge 0.$$
Moreover, if an $\epsilon$-solution is found after $\tilde{S}$ phases of the FAPL method, then we have
\begin{equation}
\label{FAPLgaprel}
\Delta_{S}> \epsilon\ge \Delta_{S+1}.
\end{equation}
Combining the above three inequalities, we conclude that
\begin{equation}
  \epsilon < q^{S-1}\Delta_1\le q^{S-1}\frac{(2R)^{1+\rho}M}{1+\rho},
\end{equation}
and part a) follows immediately from the above inequality.
We are now ready to prove part b). In view of Lemma \ref{lemq} and \eqref{FAPLgaprel}, we have
$\Delta_s\ge {\epsilon}{q^{s-S}}$. Using this estimate, part a), and Proposition \ref{prop_FAPL_gap}, we conclude that the total number of iterations performed by the FAPL method is bounded by
\begin{align}
  N(\epsilon)&=\sum_{s=1}^{S}N_s= S+\left(\frac{cMR^{1+\rho}}{(1+\rho)\theta\beta}  \right)^{\frac{2}{1+3\rho}}\sum_{s=1}^{S}q^{\frac{2(S-s)}{1+3\rho}}
  <  S+\frac{1}{1-q^{\frac{2}{1+3\rho}}}\left(\frac{cMR^{1+\rho}}{(1+\rho)\theta\beta\epsilon}  \right)^{\frac{2}{1+3\rho}},
  \end{align}
where $N_s$ denotes the number of the iterations of phase $s$ for $1\le s\le S$.
\end{proof}

In view of Theorem \ref{thmFAPL}, the FAPL method achieves the optimal iteration complexity bounds for solving nonsmooth, weakly smooth,
and smooth CP problems, which are the same as the convergence properties of the ABL and APL method (see \cite{nemyud:83,Nest04} for the discussions on the complexity theories for solving CP problems, and \cite{Lan13-1} for the convergence properties of the ABL and APL methods).

\subsection{FUSL for ball-constrained structured problems}
\label{secFUSL}
In this subsection, we still consider the ball-constrained problem in \eqnok{eqnCCP}, but assume that
its objective function is given by
\begin{equation}
\label{USLpro}
f(x):=\hat{f}(x)+F(x),
\end{equation}
where $\hat f$ is a smooth convex function, i.e., $\exists L_{\hat f}>0$ s.t.
\begin{align}
	\label{eqnLf}
	 \hat{f}(y)-\hat{f}(x)-\langle\nabla\hat{f}(x),y-x\rangle\le\frac{L_{\hat{f}}}{2}\|y-x\|^2,
\end{align}
and
\begin{equation}
\label{eqnF}
F(x):=\max\limits_{y\in Y}\{\left\langle  Ax,y\right\rangle -\hat{g}(y)\}.
 \end{equation}
Here, $Y\subseteq \mathbb{R}^m$ is a compact convex set, $\hat{g}:= Y\to \mathbb{R}$ is a relatively simple convex function,
and $A:\mathbb{R}^n\to\mathbb{R}^m$ is a linear operator. Our goal is to present a new prox-level method
for solving problem \eqnok{eqnCCP}-\eqnok{USLpro}, which can significantly reduce
the iteration cost of the USL method in  \cite{Lan13-1}.

Generally, the function $F$ given by \eqref{eqnF} is non-smooth. However,
in an important work \cite{Nest05-1}, Nesterov demonstrated that this function can be closely approximated by a class of smooth convex functions. In particular, letting $v:Y\to\R$ be a prox-function with modulus $\sigma_v$ and denoting $c_v:=\argmin_{v\in Y} v(y)$, we can approximate $F$ in \eqref{eqnF} by the smooth function
\begin{align}
F_{\eta}(x):&=\max_{y\in Y}\{\left\langle Ax,y\right\rangle -\hat{g}(y)-\eta V(y)\},\label{USLypro}
\end{align}
where $\eta>0$ is called the smoothing parameter, and $V(\cdot)$ is the Bregman divergence defined by
\begin{equation}
\label{eqnV}
V(y):=v(y)-v(c_v)-\left\langle  \nabla v(c_v),y-c_v\right\rangle .
\end{equation}
It was shown in \cite{Nest05-1} that the gradient of $F_{\eta}(\cdot)$ given by $\nabla F_{\eta}(x)= A^{\ast}y^{\ast}(x)$ is Lipschitz continuous with constant
\begin{align}
	\label{eqnLeta}
L_{\eta}:={\|A\|^2}/({\eta\sigma_v}),
\end{align}
where $\|A\|$ is the operator norm of $A$, $A^{\ast}$ is the adjoint operator, and $y^{\ast}(x)\in Y$ is the solution to the optimization problem in \eqref{USLypro}.
Moreover, the ``closeness'' of $F_{\eta}(\cdot)$ to $F(\cdot)$ depends linearly on the smoothing parameter $\eta$, i.e.,
\begin{align}
F_{\eta}(x)&\le F(x)\le F_{\eta}(x)+\eta D_{v,Y},\ \forall x\in X,
\end{align}
where
\begin{equation}
\label{defDvY}
D_{v,Y}:=\max_{y,z\in Y}\{v(y)-v(z)-\left\langle \nabla v(z),y-z\right\rangle \}.
\end{equation}
Therefore, if we denote
\begin{align}
	\label{eqnfeta}
	f_{\eta}(x):=\hat{f}(x)+F_{\eta}(x),
\end{align}
then
\begin{align}
f_{\eta}(x)&\le f(x) \le f_{\eta}(x)+\eta D_{v,Y}.\label{USLsmoothrelation}
\end{align}
Applying an optimal gradient method to minimize the smooth function $f_\eta$ in \eqref{eqnfeta}, Nesterov proves
in \cite{Nest05-1} that the iteration complexity for computing an $\epsilon$-solution to
problem \eqnok{eqnCCP}-\eqref{USLpro} is bounded by ${\cal O}(1/\epsilon)$. However,
the values of quite a few problem parameters, such as $\|A\|,\sigma_v$ and $D_{v,Y}$, are required for the implementation of Nesterov's smoothing scheme.

By incorporating Nesterov's smoothing technique \cite{Nest05-1} into the APL method, Lan developed in \cite{Lan13-1}
a new bundle-level type method, namely the uniform smoothing level (USL) method, to solve structured problems
given in the form of \eqref{USLpro}.
%The USL method can be viewed as the application of the APL method on minimizing the approximated smooth function $f_{\eta}(x)$.
While the USL method achieves the same optimal iteration complexity  as Nesterov's smoothing scheme in \cite{Nest05-1}, one advantage of the USL method over Nesterov's smoothing scheme is that the smoothing parameter $\eta$ is adjusted dynamically during the execution, and an estimate of $D_{v,Y}$ is obtained automatically, which makes the USL method problem parameter free. However, similar to the APL method, each iteration of the USL method involves the solutions of two subproblems.
Based on the USL method in \cite{Lan13-1} and our analysis of the FAPL method in Section \ref{secFAPL}, we propose a fast USL (FUSL) method that solves problem \eqref{USLpro} with the same optimal iteration complexity as the USL method, but requiring only to solve one simpler subproblem in each iteration.

\vgap

Similar to the FAPL method, the FUSL method consists of different phases, and each phase calls a gap reduction procedure,
denoted by $\mathcal{G}_{FUSL}$,  to reduce the gap between the upper and lower bounds on $f^{\ast}_{\overline{x},R}$ in \eqref{eqnCCP} by a constant
factor. We start by describing procedure $\mathcal{G}_{FUSL}$.

\begin{procedure}[h]
\caption{\label{funFUSL}The FUSL gap reduction procedure: $(x^+,D^+, \lb^+)=\mathcal{G}_{FUSL}(\hat{x},D,\lb, R,\overline{x},\beta,\theta)$}
\begin{algorithmic}[1]
		\makeatletter
		\setcounter{ALG@line}{-1}
		\makeatother
\State \label{stepGFUSLinit}Let $k=1$, $\overline{f}_0=f(\hat{x}), l=\beta \cdot \lb+(1-\beta)\overline{f}_0$, $Q_0=\R^n$, $x_0^u=\hat{x}$, $x_0 \in B(\overline{x}, R)$ be arbitrarily given,
and
\begin{equation}
\label{etadef}
\eta:=\theta(\overline{f}_0-l)/(2D).
\end{equation}
\State \label{stepGFUSLcut}Update the cutting plane model: set $x_k^l$ to \eqref{def_xkl}, $\underline Q_k$ to \eqnok{DEFQ_kLow}, and
\begin{align}
h(x_k^l,x)=h_{\eta}(x_k^l,x)=f_{\eta}(x_k^l)+\left \langle f'_{\eta}(x_k^l),x-x_k^l \right \rangle. \label{new_cut}
\end{align}
\State \label{stepGFUSLp}Update the prox-center: set $x_k$ to \eqnok{FAPLsub}.
If $\underline{Q}_k=\emptyset$ or $\|x_k-\overline{x}\|>R$, then \textbf{terminate} with output $x^+=x^u_{k-1},D^+=D,\lb^+=l$.

\State \label{stepGFUSLub}Update the upper bound and the estimate of $D_{v,Y}$: set $\tilde x_k^u$ to \eqnok{eqnxku}, $x_k^u$ to \eqnok{eqnxku1},
and $\overline{f}_k=f(x_k^u)$. Check the following conditions:
\begin{enumerate}[3a)]
\item \label{stepGFUSLD}if $f(x_k^u)\le l+\theta(\overline{f}_0-l)$, then \textbf{terminate} with output $x^+=x^u_{k},D^+=D,\lb^+=\lb$.
\item \label{stepGFUSL2D}if $f(x_k^u)> l+\theta(\overline{f}_0-l)$ and $f_{\eta}(x_k^u)\le l +\frac{\theta}{2}(\overline{f}_0-l)$, then \textbf{terminate} with output $x^+=x^u_{k}, D^+=2D, \lb^+=\lb$.
\end{enumerate}
\State \label{stepGFUSLQ} Choose $Q_k$ as same as Step \ref{stepGFAPLQ} in $\mathcal{G}_{FAPL}$, set $k=k+1$, and go to step 1.
\end{algorithmic}
\end{procedure}

A few remarks about  procedure $\mathcal{G}_{FUSL}$ are in place.
Firstly,  since the nonsmooth objective function $f$ is replaced by its smoothed approximation $f_\eta$,
we replace the cutting plane model in \eqref{linearappro} with the one for $f_\eta$ (see \eqref{new_cut}). Also note that for the USL method in \cite{Lan13-1}, $\hat f$ is assumed to be a simple Lipschitz continuous convex function, and only $F_{\eta}$ is approximated by the linear estimation. However in the FUSL method, we assume $\hat f$ is general smooth convex, and linearize both $\hat f$ and $F_{\eta}$ in \eqref{new_cut}.
Secondly, the smoothing parameter $\eta$ is specified as a function of the parameter $D$, $\bar f_0$ and $l$,
where $D$ is an estimator of $D_{v,Y}$ in \eqref{defDvY} and given as an input parameter to
procedure  $\mathcal{G}_{FUSL}$. Thirdly,
same to the FAPL method, the parameters $\{\alpha_k\}$ are chosen according to \eqref{stepsize}. Such conditions
are required to guarantee the optimal convergence of the FUSL method for solving problem \eqnok{eqnCCP}-\eqref{USLpro}.
Fourthly, similar to the FAPL method, the feasible sets $\underline{Q}_k, Q_k, \overline{Q}_k$
only contains a limited number of linear constraints, and
there is only one subproblem (i.e. \eqref{FAPLsub}) involved in procedure $\mathcal{G}_{FUSL}$,
which can be solved exactly when the depth of memory is small.

The following lemma provides some important observations about procedure $\mathcal{G}_{FUSL}$, which
are similar to those for the USL gap reduction procedure in \cite{Lan13-1}.
\begin{lemma}
\label{lemSigPhase}
The following statements hold for procedure $\mathcal{G}_{FUSL}$.
\begin{enumerate}[a)]
\item If this procedure terminates at steps \ref{stepGFUSLp} or \ref{stepGFUSLD}, then we have $\ub^+-\lb^+\le q(\ub-\lb)$, where $q$ is defined in \eqref{qdef} and $\ub:=\overline{f}_0,\ub^{+}:=f(x^{+})$.
\item If this procedure terminates at step \ref{stepGFUSL2D}, then $D< D_{v,Y}$ and $D^+<2D_{v,Y}$.
\end{enumerate}
\end{lemma}

\begin{proof}
The proof of part a) is as the same as that of Lemma \ref{lemq}, and we only show part b) here. By the termination condition at step \ref{stepGFUSL2D}, we have $f(x_k^u)>l+\theta(\ub-l)$ and $f_{\eta}(x_k^u)\le l+\frac{\theta}{2}(\ub-l)$. So,
$$f(x_k^u)-f_{\eta}(x_k^u)>\frac{\theta}{2}(\ub-l).$$
We conclude from the above relation, \eqref{USLsmoothrelation}, and \eqref{etadef} that
$$D_{v,Y}\ge\frac{f(x_k^u)-f_{\eta}(x_k^u)}{\eta}>\frac{\theta(\ub-l)}{2\eta} = D.$$
Finally, $D^+<2D_{v,Y}$ comes immediately from the above relation and the definition of $D^+$ in step \ref{stepGFUSL2D}.
\end{proof}

The following results provides a bound on the number of iterations performed by procedure $\mathcal{G}_{FUSL}$.
\begin{proposition}
\label{Prop_FUSL_GAP}
Suppose that $\{\alpha_k\}_{k\ge 1}$ in procedure $\mathcal{G}_{FUSL}$ are chosen such that \eqref{stepsize} holds.  Then,
the number of iterations performed by  this procedure does not exceed
\begin{equation}
\label{FUSLiteration}
\overline{N}(\Delta, D):=R\sqrt{\frac{cL_{\hat{f}}}{\theta\beta\Delta}}+\frac{\sqrt{2}R\|A\|}{\theta\beta\Delta}\sqrt{\frac{cD}{\sigma_v}}+1,
\end{equation}
where $\Delta:=f(\hat x)-\lb$.
\end{proposition}

\begin{proof}
It is easy to see that the gradient of $f_{\eta}$ in \eqref{USLpro} has Lipschitz continuous gradient with constant $L=L_{\hat{f}}+L_{\eta}$, where $L_\eta$ and $L_{\hat{f}}$ are defined in \eqref{eqnLeta} and \eqref{eqnLf}, respectively.
Suppose that procedure $\mathcal{G}_{FUSL}$ does not terminate at step $k$.
Noting that the prox-function $d(x)$ in procedure $\mathcal{G}_{FUSL}$ has modulus $1$, similarly to the discussion on \eqref{pfrelation}, we have
\begin{equation}
\label{USLpfrelation}
f_{\eta}(x_k^u)-l\le\frac{cL d(x_k)}{k^2}\le\frac{cLR^2}{2k^2},
\end{equation}
where $c$ is defined in \eqref{stepsize}, and the second inequality is from \eqnok{Bounddk}.
Also, since procedure $\mathcal{G}_{FUSL}$ does not terminate, in view of the termination condition at step \ref{stepGFUSL2D} and the definition of $l$ in step \ref{stepGFUSLinit}, we have
\begin{equation}
f_{\eta}(x_k^u)-l>\frac{\theta\beta\Delta}{2}.
\end{equation}
Combining the above two relations, and noting \eqref{eqnLeta} and \eqref{etadef}, we conclude that
\begin{equation}
k\le\sqrt{\frac{cLR^2}{\theta\beta\Delta}}\le R\sqrt{\frac{cL_{\hat{f}}}{\theta\beta\Delta}}+\frac{\sqrt{2
}R\|A\|}{\theta\beta\Delta}\sqrt{\frac{cD}{\sigma_v}}.
 \end{equation}
\end{proof}

We are now ready to describe the FUSL method which iteratively calls procedure  $\mathcal{G}_{FUSL}$ to solve
the structured saddle point problem \eqnok{eqnCCP}-\eqnok{USLpro}.

\begin{algorithm}[h]
	\caption{\label{algFUSL}{The fast uniform smoothing level (FUSL) method}}
\begin{algorithmic}[1]
		\makeatletter
		\setcounter{ALG@line}{-1}
		\makeatother
\State \label{stepFUSLinit}Given ball $B(\overline{x}, R)$, choose initial point $p_0\in B(\overline{x}, R)$, prox-function $v(\cdot)$ for the smoothing function $F_{\eta}$ in \eqref{USLypro} and \eqref{eqnV}, initial guess $D_1$ on the size $D_{v,Y}$ in \eqref{defDvY}, tolerance $\varepsilon>0$, and parameters $\beta,\theta \in(0,1)$.
\State \label{stepFUSLinitbound}Set $p_1\in \Argmin_{x\in B(\overline{x}, R)}\ h(p_0,x)$, $\lb_1=h(p_0,p_1), \ub_1=\min\{f(p_0),f(p_1)\}$, let $\hat{x}_1$ be either $p_0$ or $p_1$ such that $f(\hat{x}_1)=\ub_1$,  and $s=1$.
\State \label{stepFUSLmain}If $\ub_s-\lb_s\leq \epsilon$, \textbf{terminate} and \textbf{output} approximate solution $\hat{x}$.
\State \label{stepFUSLcall}Set $(\hat{x}_{s+1},D_{s+1}, \lb_{s+1})=\mathcal{G}_{FUSL}(\hat{x}_s, D_s,  \lb_s, R, \overline{x},\beta,\theta)$ and $\ub_{s+1}=f(\hat{x})$.
\State Set $s=s+1$ and go to step \ref{stepFUSLmain}.

\end{algorithmic}
\end{algorithm}

Similar to the FAPL method, we say that a phase of the FUSL method occurs when $s$ increases by $1$.
More specifically,
similar to the USL method, we classify two types of phases in the FUSL method.
A phase is called \emph{significant} if the corresponding $\mathcal{G}_{FUSL}$ procedure terminates
at steps \ref{stepGFUSLp} or \ref{stepGFUSLD}, otherwise it is called \emph{non-significant}.
Clearly, if the value of $D_{v,y}$ is provided, which is the assumption made in Nesterov's smoothing scheme \cite{Nest05-1},
then we can set $D_1=D_{v,Y}$ in the scheme of both the original and modified FUSL method, and consequently, all the phases of both the original and modified FUSL methods become significant.

For the sake of simplicity, an iteration of procedure $\mathcal{G}_{FUSL}$ is also referred to an iteration of the
FUSL method. The following result establishes a bound on
the total number of iterations performed by the FUSL method to find an $\epsilon$-solution
of problem \eqnok{eqnCCP}-\eqnok{USLpro}.  Note that the proof of these results
is similar to that of Theorem 7 in \cite{Lan13-1}.
%We will calculates the inner iterations performed by significant phases and non-significant phases separately to finish the convergence analysis for the FUSL method.

\begin{theorem}
\label{thmFUSL}
Suppose that $\{\alpha_k\}$ in procedure $\mathcal{G}_{FUSL}$ are chosen such that \eqref{stepsize} holds. Then,
the total number of iterations performed by the FUSL method for computing an $\epsilon$-solution of problem \eqnok{eqnCCP}-\eqnok{USLpro} is bounded by
\begin{equation}\label{eqnFUSL}
\overline{N}(\epsilon):=S_1+S_2+(\frac{2}{\sqrt 2 -1}+\frac{\sqrt 2}{1- q} )\frac{R\|A\|}{\theta\beta\epsilon}\sqrt{\frac{c\tilde{D}}{\sigma_v}}
+(S_1+\frac{1}{1-\sqrt q})R\sqrt{\frac{cL_{\hat f}}{\theta\beta\epsilon}},
\end{equation}
where $q$ and $D_{v,Y}$ are defined in \eqref{qdef} and \eqref{defDvY} respectively, and
\begin{equation}
\tilde{D}:=\max\{D_1,2D_{v,Y}\}, S_1:=\max\left\{\left\lceil \log_2\frac{D_{v,Y}}{D_1} \right\rceil, 0\right\} \text{\ and\ }S_2:=\left\lceil \log_{\frac{1}{q}}\frac{4\sqrt{2}R\|A\|\sqrt{\frac{D_{v,Y}}{\sigma_v}}+2R^2L_{\hat{f}}}{\epsilon} \right\rceil.
\label{defS1S2D}
\end{equation}
% \text{\ \ and\ \ } \hat{D}:=\max\{D_0,2D_{v,Y}\}.
%\begin{equation}
%\label{defS1}
%S_1\le\max\left\{\left\lceil \log_2\frac{D_{v,Y}}{D_0} \right\rceil, 0\right\}.
%\end{equation}
%\begin{equation}
%\label{defS2}
%  S_2:=\left\lceil \log_{\frac{1}{q}}\frac{\Delta_1}{\epsilon} \right\rceil,
%\end{equation}

\end{theorem}
\begin{proof}
We prove this result by estimating the numbers of iterations performed within both non-significant and significant phases. Suppose that the set of indices of the non-significant and significant phases are  $\{m_1,m_2,\ldots,m_{s_1}\}$ and $\{n_1,n_2,\ldots,n_{s_2}\}$ respectively. For any non-significant phase $m_k$, $1\le k\le s_1$, we can easily see from step \ref{stepGFUSL2D} that $D_{m_{k+1}}=2D_{m_k}$, by part b) in Lemma \ref{lemSigPhase}, the number of non-significant phases performed by the FUSL method is bounded by $S_1$ defined above, i.e., $s_1\le S_1$.

In addition, since $D_{m_{s_1}}\le \tilde{D}$, we have $D_{m_k}\le (1/2)^{s_1-k}\tilde{D}$, where $\tilde{D}$ is defined above. Combining the above estimates on $s_1$ and $D_{m_k}$, and in view of the fact $\Delta_{m_k}>\epsilon$ for all $1\le k\le s_1$,  we can bound the number of iterations performed in non-significant phases by
\begin{align}
\label{defN1}
\begin{aligned}
\overline{N}_1&=\sum_{k=1}^{s_1}\overline{N}(\Delta_{m_k},D_{m_k})\le\sum_{k=1}^{s_1}\overline{N}\left(\epsilon,\tilde{D}/2^{s_1-k}\right)\\
              & \le S_1\left(R\sqrt{\frac{cL_{\hat{f}}}{\theta\beta\epsilon}}+1\right)+\frac{\sqrt{2}R\|A\|}{\theta\beta\epsilon}\sqrt{\frac{c}{\sigma_v}}\sum_{k=1}^{S_1}\sqrt{\frac{\tilde{D}}{2^{S_1-k}}}\\
              &\le S_1\left(R\sqrt{\frac{cL_{\hat{f}}}{\theta\beta\epsilon}}+1\right)+\frac{2R\|A\|}{(\sqrt{2}-1)\theta\beta\epsilon}\sqrt{\frac{c\tilde{D}}{\sigma_v}}.
\end{aligned}
\end{align}
Applying Lemma 8 in \cite{Lan13-1} and relation \eqref{eqnLf}, and in view of the fact that $p_0, p_1\in B(\overline{x},R)$ in Algorithm \ref{algFUSL}, the initial gap is bounded as
\begin{align}
\Delta_1&:=\ub_1-\lb_1\le\left[F(p_0)-F(p_1)-\left\langle F'(p_1),p_0-p_1\right\rangle \right]+\left[\hat{f}(p_0)-\hat{f}(p_1)-\left\langle \hat{f}'(p_1),p_0-p_1\right\rangle \right]\\
&\le 4\sqrt{2}R\|A\|\sqrt{\frac{D_{v,Y}}{\sigma_v}}+2R^2L_{\hat{f}},
\end{align}
where $F'(p_1)\in \partial F(p_1)$. Then for significant phases, similarly to the proof of Theorem \ref{thmFAPL}, we have $s_2\le S_2$. Moreover, for any $n_k$, $1\le k\le s_2$, using Lemmas \ref{lemq}, \ref{lemSigPhase}, we have $D_{n_k}\le \tilde{D}$, $\Delta_{n_{k+1}}\le q\Delta_{n_{k}}$, and $\Delta_{n_{s_2}}> \epsilon$, which implies $\Delta_{n_k}>\epsilon/q^{s_2-k}$.
Combining such an estimate on $D_{n_k}, \Delta_{n_k}$ and bound on $s_2$, we can see that the total number of iterations performed the significant phases is bounded by
\begin{align}
\label{defN2}
\begin{aligned}
\overline{N}_2&=\sum_{k=1}^{s_2}\overline{N}(\Delta_{n_k},D_{n_k})\le\sum_{k=1}^{s_2}\overline{N}(\epsilon/q^{s_2-k}, \tilde{D})\\
              &\le S_2+R\sqrt{\frac{cL_{\hat{f}}}{\theta\beta\epsilon}}\sum_{k=1}^{S_2}q^{\frac{S_2-k}{2}}+\frac{\sqrt{2}R\|A\|}{\theta\beta\epsilon}\sqrt{\frac{c \tilde{D}}{\sigma_v}}\sum_{k=1}^{S_2}q^{S_2-k}\\
              & \le S_2+\frac{R}{1-\sqrt{q}}\sqrt{\frac{cL_{\hat{f}}}{\theta\beta\epsilon}}+\frac{\sqrt{2}R\|A\|}{\theta\beta\epsilon(1-q)}\sqrt{\frac{c\tilde{D}}{\sigma_v}}.
\end{aligned}
\end{align}
Finally, the total number of iterations performed by the FUSL method is bounded by $\overline{N}_1+\overline{N}_2$, and thus \eqref{eqnFUSL} holds.
\end{proof}

\vgap

From \eqref{eqnFUSL} in the above theorem, we can see that the iteration complexity of the FUSL method for solving problem \eqnok{eqnCCP}-\eqnok{USLpro} is bounded by
\begin{align}
	\cO\left(\sqrt{\frac{L_{\hat{f}}}{\epsilon}} + \frac{\|A\|}{\epsilon}\right).
\end{align}
The above iteration complexity is the same as that of the Nesterov smoothing scheme in \cite{Nest05-1} and the USL method in \cite{Lan13-1}. However, both the USL and FUSL methods improve Nesterov's smoothing scheme in that both of them are problem parameter free.
In addition, as detailed in Subsection~\ref{secKKT} below, the FUSL method further improves the USL method by
significantly reducing its iteration cost and improving the accuracy for solving its subproblems.

\subsection{Solving the subproblems of FAPL and FUSL}
\label{secKKT}
In this section, we introduce an efficient method to solve the subproblems \eqref{FAPLsub} in the FAPL and FUSL methods, which are given in the form
of
\begin{equation}
\label{subsolver}
x_c^{\ast}:=\argmin_{x\in Q}\frac{1}{2}\|x-p\|^2.
\end{equation}
Here, $Q$ is a closed polyhedral set described by $m$ linear inequalities, i.e.,
$$Q:=\{x\in \R^n: \left\langle A_i,x\right\rangle \le b_i, \  i=1,2,\ldots,m\}.$$

Now let us examine the Lagrange dual of \eqnok{subsolver} given by
\begin{equation}
\max_{\lambda\ge 0} \min_{x \in \bbr^n} \frac{1}{2}\|x-p\|^2+\sum_{i=1}^{m}\lambda_i[\left\langle A_{i},x\right\rangle -b_i]. \label{Lag_dual0}
\end{equation}
It can be checked from the theorem of alternatives that
problem \eqref{Lag_dual0} is solvable if and only if $Q \neq \emptyset$.
Indeed, if $Q \neq \emptyset$, it is obvious that the optimal value of \eqnok{Lag_dual0}
is finite. On the other hand, if $Q = \emptyset$, then there exists $\bar{\lambda} \ge 0$ such that
$\bar{\lambda}^T A = 0$ and $\bar {\lambda}^T b < 0$, which implies that
the optimal value of \eqnok{Lag_dual0} goes to infinity. Moreover,
if \eqnok{Lag_dual0} is solvable and $\lambda^*$ is one of its optimal dual solutions,
then
\begin{equation}
\label{subxsol}
x_c^{\ast}=p-\sum\limits_{i=1}^{m}\lambda_i^{\ast}A_i.
\end{equation}
It can also be easily seen that \eqnok{Lag_dual0} is equivalent to
\beq
\max_{\lambda \ge 0} - \frac{1}{2}\lambda^T M \lambda + C^T\lambda, \label{lambdapro}
\eeq
where
$
M_{ij}:=\left\langle A_i,A_j\right\rangle ,\ C_i:=\left\langle A_i,p\right\rangle -b_i, \ \forall i,j=1,2,\ldots,m.
$
Hence, we can determine the feasibility of \eqnok{subsolver} or compute its optimal solution
by solving the relatively simple problem in \eqnok{lambdapro}.

%Noting that $A_i,A_j\ (i\neq j)$ are linearly independent, we can see that $M$ is positive definite, and thus the problem \eqref{lambdapro} is strongly convex.
Many algorithms are capable of solving the above nonnegative quadratic programming in \eqnok{lambdapro} efficiently.
Due to its low dimension (usually less than $10$ in our practice), we propose a brute-force method to compute
the exact solution of this problem.
Consider the Lagrange dual associated with \eqref{lambdapro}:
$$\min_{\lambda\ge 0}\max_{\mu\ge 0 }\cL(\lambda,\mu):= \frac{1}{2}\lambda^T M \lambda - (C^T+\mu)\lambda,$$
where the dual variable is $\mu:=(\mu_1,\mu_2,\ldots,\mu_m)$. Applying the KKT condition, we can see that $\lambda^{\ast}\ge 0$ is a solution of problem \eqref{lambdapro} if and only if there exists $\mu^*\ge 0$ such that
\begin{align}
\label{lambdamupro}
\nabla_{\lambda} \cL(\lambda^{\ast},\mu^{\ast})=0 \ \ \ \mbox{and} \ \ \  \langle \lambda, \mu\rangle=0.
\end{align}
Note that the first identity in \eqref{lambdamupro} is equivalent to a linear system:
\begin{align}
\begin{pmatrix}
\label{sublinearsystem}
M & -I\end{pmatrix}\begin{pmatrix}\lambda_1\\ \vdots\\\lambda_m\\\mu_1\\ \vdots\\ \mu_m\end{pmatrix}=\begin{pmatrix}b_1\\b_2\\ \vdots\\ b_m
\end{pmatrix},
\end{align}
where $I$ is the $m\times m$ identity matrix. The above linear system has $2m$ variables and $m$ equations. But for
any $i=1,\ldots,m$, we have either $\lambda_i=0$ or $\mu_i=0$, and hence we only need to consider $2^m$ possible cases
on the non-negativity of these variables. Since $m$ is rather small in practice, it is possible to exhaust all these $2^m$ cases to
find the exact solution to \eqref{lambdamupro}. For each case, we first remove the $m$ columns in the matrix $(M \ -I)$ which correspond to the $m$ variables assumed to be $0$, and then solve the remaining determined linear system. If all variables of the
computed solution are non-negative, then solution $(\lambda^{\ast},\mu^{\ast})$ to \eqref{lambdamupro} is found, and the exact solution $x_c^*$ to \eqref{subsolver} is computed by \eqref{subxsol}, otherwise, we continue to examine the next case.
It is interesting to observe that these different cases can also be considered in parallel
to take the advantages of high performance computing techniques.

\subsection{Extending FAPL and FUSL for unconstrained problems}
\label{secunconstraint}
In this subsection,
we study how to utilize the FAPL and FUSL method to solve
the unconstrained problems based on our results in Section~\ref{secCCP}.

Let us first consider the case when $f$ in \eqnok{oriproblem} satisfies \eqnok{smoothrelation}.
If the method $\mathcal{A}$ in step 1 of Algorithm~\ref{algOracle} is given by
the FAPL method, then by Theorem \ref{thmFAPL}, the number of evaluations
of $f'$ within one call to $\mathcal{A}(\overline{x},2r_k,\Delta_k)$ is bounded by
\begin{align}
	 \left[\frac{cM_k(2r_k)^{1+\rho_k}}{(1+\rho_k)\theta\beta\Delta_k}\right]^{\frac{2}{1+3\rho_k}},
\end{align}
where $c$ is a universal constant,
	$M_k:=M(B(\overline{x},2r_k))$ and $\rho_k:=\rho(B(\overline{x},2r_k))$
are constants corresponding to the assumption in \eqref{smoothrelation}.
By Theorem \ref{unconstraintheorem}, the number of evaluations of $f'$ up to the $k$-th iteration of Algorithm \ref{algOracle} is bounded by
\begin{align}
	 \cO\left(\left[\frac{M(4D^{\ast})^{1+\rho}}{\epsilon_k}\right]^{\frac{2}{1+3\rho}}\right),
\end{align}
where
$M:=M(B(\overline{x},4D^{\ast}))$, $\rho:=\rho(B(\overline{x},4D^{\ast}))
$ and $\epsilon_k:=f(x_k)-f^{\ast}$ is the accuracy of the solution.
%Therefore, Algorithm \ref{algOracle} achieves the optimal iteration complexity for uniformly solving unconstrained non-smooth, smooth and weakly smooth CP of form \eqref{oriproblem}.
It should be noted that the constants $M$ and $\rho$ are local constants that depend on the distance from $\overline{x}$ and $x^*$, which are not required for the FAPL method and Algorithm \ref{algOracle}, and also generally smaller than the constants $M(\R^n)$ and $\rho(\R^n)$, respectively, for the global H\"older continuity condition.

Moreover, if $f$ in \eqref{oriproblem} is given in the form of \eqref{USLpro} as a structured nonsmooth CP problem, then the FUSL method could be applied to solve the corresponding structured ball-constraint problem in Algorithm \ref{algOracle}. By Theorem \ref{thmFUSL}, the number of evaluations of $f'$ within one call to $\mathcal{A}(\overline{x},2r_k,\Delta_k)$ is bounded by
\begin{equation}
 S_1+S_2+2r_kC'\sqrt{\frac{L_{\hat{f}}}{\Delta_k}}+\frac{2r_kC''\|A\|}{\Delta_k},
\end{equation}
where $C',C''$ are some constants depending on the parameters $q,\theta,\beta,\sigma_v,D_0$ and $D_{v,Y}$ in the FUSL method.
%S_2+\frac{2r_k}{1-\sqrt{q}}\sqrt{\frac{cL_{\hat{f}}}{\theta\beta\Delta_k}}+\frac{2\sqrt{2}r_k\|A\|}{\theta\beta\Delta_k(1-q)}\sqrt{\frac{cD_{v,Y}}{\sigma_v}}.

Applying Theorem \ref{unconstraintheorem} with $\alpha_1=\alpha_2 =1$, $\beta_1=\frac{1}{2}$, $\beta_2=1$, %$C_1(\overline{x},R,f)=\sqrt{cL_{\hat{f}}/(\theta\beta)}/ (1-\sqrt{q})$,  and $C_2(\overline{x},R,f)=\sqrt{2}\|A\| \sqrt{cD_{v,Y}/\sigma_v} / (\theta\beta(1-q))$,
$C_1(\overline{x},R,f)=2C' \sqrt{L_{\hat{f}}}$,  and $C_2(\overline{x},R,f)=2C''\|A\|$,
the number of evaluations of $f'$ up to the $k$-th iteration of Algorithm \ref{algOracle} is bounded by
%\begin{equation}
%  \mathcal{O}\left( \frac{4D^{\ast}}{1-\sqrt{q}}\sqrt{\frac{cL_{\hat{f}}}{\theta\beta\epsilon_k}}+\frac{4\sqrt{2}D^{\ast}\|A\|}{\theta\beta\epsilon_k(1-q)}\sqrt{\frac{cD_{v,Y}}{\sigma_v}}     \right)
%\end{equation}
\begin{equation}
  \mathcal{O}\left( 4D^{\ast}C'\sqrt{\frac{L_{\hat{f}}}{\epsilon_k}}+\frac{4C''D^{\ast}\|A\|}{\epsilon_k} \right).
\end{equation}

Similar to the FAPL method, here $L_f:=L_f(B(\overline{x},4D^{\ast}))$ is a lower bound of $L_f(\R^n)$.
%Similar to the FAPL method, here $L_f:=L_f(B(x_0,4\|x^{\ast}-x_0\|)), \|A\|:=\|A\|(B(x_0,4\|x^{\ast}-x_0\|))$ is a relaxation of $L_f(\R^n)$ and $\|A\|(\R^n)$.

\setcounter{equation}{0}
\section{Generalization to strongly convex optimization}
\label{secStronglyconvex}
In this section, we generalize the FAPL and FUSL methods for solving convex optimization problems
in the form of \eqnok{oriproblem} whose objective function $f$ satisfies %\begin{equation}
%\label{stronglyprob}
%f^{\ast}:=\min_{x\in R^n} f(x)
%\end{equation}
%In particular, in addition to \eqref{smoothrelation}, we assume that the objective function $f(\cdot)$ in \eqref{oriproblem} is strongly convex with module $\mu>0$, i.e.,
\begin{equation}
\label{eqnmu}
f(y)-f(x)-\left\langle f'(x),y-x\right\rangle \ge \frac{\mu}{2}\|y-x\|^2, \  \ \forall x,y \in \R^n,
\end{equation}
for some $\mu > 0$.
For the sake of simplicity, we assume throughout this section that an initial
lower bound $\lb_0 \le f^{\ast}$ is available\footnote{Otherwise, we should incorporate a guess-and-check procedure
similar to the one in Section~\ref{secCCP}.}. Under this assumption, it follows from
\eqref{eqnmu} that $\|p_0 - x^*\|^2 \le 2 [f(p_0) - \lb_0] / \mu$ for a given initial point $p_0$,
and hence that the FAPL and FUSL methods for ball-constrained problems can be directly applied. However, since
the lower and upper bounds on $f^*$ are constantly improved
in each phase of these algorithms, we can shrink the ball constraints
by a constant factor once every phase accordingly. We show that
the rate of convergence of the FAPL and FUSL methods can be significantly
improved in this manner.

We first present a modified FAPL method for solving black-box CP problems which satisfy both \eqnok{smoothrelation}
and \eqnok{eqnmu}. More specifically, we modify the ball constraints used in the FAPL method by
shifting the prox-center $\bar x$ and shrinking the radius $R$ in procedure ${\mathcal{G}}_{FAPL}$.
Clearly, such a modification does not incur any extra computational cost.
This algorithm is formally described as follows.

\begin{procedure}[h]
\caption{\label{funFAPLsc} The modified FAPL gap reduction procedure: $(x^+, \lb^+)=\tilde{\mathcal{G}}_{FAPL}(\hat{x}, \lb, r, \beta,\theta)$}
In Procedure \ref{funFAPL}, set $\overline{x}=\hat{x}$, and consequently the prox-function $d$ in \eqnok{FAPLsub} is replaced by $\|x-\hat{x}\|^2/2$.
\end{procedure}

\begin{algorithm}[h]
	\caption{\label{algFAPLsc}{The modified FAPL method  for minimizing strongly convex functions}}
	In Algorithm \ref{algFAPL}, change steps \ref{stepFAPLinit}, \ref{stepFAPLinitbound} and \ref{stepFAPLcall} to
	\begin{algorithmic}[1]
		\makeatletter
		\setcounter{ALG@line}{-1}
		\makeatother
		\State \label{stepFAPLinit_sc}Choose initial lower bound $\lb_1\le f^*$, initial point $p_0\in \R^n$, initial upper bound $\ub_1=f(p_0)$, tolerance $\epsilon>0$ and parameters $\beta,\theta \in(0,1)$.
        \State Set $\hat{x}_1=p_0$, and $s=1$.
		\makeatletter
		\setcounter{ALG@line}{2}
		\makeatother
		\State Set $(\hat{x}_{s+1}, \lb_{s+1})=\tilde{\mathcal{G}}_{FAPL}(\hat{x}_{s}, \lb_s,\sqrt{{2(f(\hat x_s)-\lb_s)}/{\mu}},\beta,\theta)$ and  $\ub_{s+1}=f(\hat{x}_{s+1})$.
	\end{algorithmic}
\end{algorithm}

A few remarks on the above modified FAPL method are in place.
Firstly, let $x^*$ be the optimal solution of problem \eqref{oriproblem} and define $\Delta_s = \ub_s -\lb_s$.
By the definition of $\ub_s$ and $\lb_s$, we have
$f(\hat{x}_s)-f(x^{\ast})\le \Delta_s$, which, in view of \eqnok{eqnmu},
then implies that
\begin{equation}
\label{strongconvexconstrain}
\|\hat{x}_s-x^{\ast}\|^2\le \frac{2\Delta_s}{\mu}=:r^2,
\end{equation}
and $x^{\ast}\in B(\hat{x}_s,r)$.
Secondly,
similar to procedure ${\mathcal{G}}_{FAPL}$, if procedure $\tilde {\mathcal{G}}_{FAPL}$ terminates at step 2, we have $\cE_f(l)\cap B(\hat{x}_s,r)=\emptyset$. Combining this with the fact $x^{\ast}\in B(\hat{x}_s,r)$, we conclude that $l$ is a valid lower bound on $f^{\ast}$. Therefore, no matter whether procedure $\tilde {\mathcal{G}}_{FAPL}$ terminates at step 2 or step 4, the gap between upper and lower bounds on $f^{\ast}$ has been reduced and $\Delta_{s+1}\le q\Delta_s$, where the $q$ is defined in \eqref{qdef}.

\vgap

We establish in Theorem~\ref{thmStrongFAPL} the iteration complexity bounds of the
modified FAPL method for minimizing strongly convex functions.

\begin{theorem}
\label{thmStrongFAPL} Suppose that $\{\alpha_k\}_{k\ge 1}$ in procedure $\tilde{\mathcal{G}}_{FAPL}$ are chosen such that \eqref{stepsize} holds.
Then the total number of iterations performed by the modified FAPL method for computing an $\epsilon$-solution
of problem \eqnok{oriproblem} is bounded by\\
\[
\widetilde{S}\left(\sqrt{\frac{2cM}{\theta\beta\mu}}+1\right) \ \ \
\mbox{and} \ \ \
\widetilde{S}+  \frac{1}{1-q^{\frac{1-\rho}{1+3\rho}}}\left( \frac{4^{1+\rho}cM} {\theta\beta(1+\rho)\mu^{\frac{1+\rho}{2}}\epsilon^{\frac{1-\rho}{2}}}   \right)^{\frac{2}{1+3\rho}},
\]
respectively, for smooth strongly convex functions (i.e., $\rho = 1$) and
nonsmooth or weakly smooth strongly convex functions (i.e., $\rho \in [0,1)$),
where $q$ is defined in \eqref{qdef}, $\lb_1$ and $\ub_1$ are given initial lower bound and upper bound
on $f^*$, and
\begin{equation}
\label{eqnFAPLsc_phases}
\widetilde{S}:=\left\lceil \log_{\frac{1}{q}}\left( \frac{\ub_1-\lb_1}{\epsilon}\right) \right\rceil.
\end{equation}
\end{theorem}

\begin{proof}
Suppose that procedure $\tilde{\mathcal{G}}_{FAPL}$ does not terminate at the $k^{th}$ inner iteration. It then follows from \eqref{pfrelation} and
\eqref{strongconvexconstrain} that
 \begin{equation}
f(x_k^u)-l\le\frac{Mr^{1+\rho}}{1+\rho}\cdot \frac{c}{k^{\frac{1+3\rho}{2}}}.
\end{equation}
Moreover,
in view of the termination condition at steps \ref{stepGFAPLub} and relation \eqref{strongconvexconstrain}, we have
$f(x_k^u)-l\ge \theta(\ub_s-l)=\theta\beta\Delta_s$  and  $r=\sqrt{2\Delta_s/\mu}$. Combining all the above observations we conclude that
\begin{equation}
  k\le \left( \frac{2^{\frac{1+\rho}{2}}cM} {\theta\beta(1+\rho)\mu^{\frac{1+\rho}{2}}\Delta_s^{\frac{1-\rho}{2}}}   \right)^{\frac{2}{1+3\rho}}.
\end{equation}
So the number of inner iterations performed in each call to  procedure $\tilde{\mathcal{G}}_{FAPL}$ is bounded by
\begin{equation}
\left( \frac{2^{\frac{1+\rho}{2}}cM} {\theta\beta(1+\rho)\mu^{\frac{1+\rho}{2}}\Delta_s^{\frac{1-\rho}{2}}}   \right)^{\frac{2}{1+3\rho}}+1.
\end{equation}
Since the gap between the upper and lower bounds on $f^{\ast}$ is reduced by a constant factor in each phase, i.e., $\Delta_{s+1}\le q\Delta_s$,
it easy to see that the total number of phases is bounded by $\tilde{S}$ defined above. Using
the previous two conclusions and the fact that $\Delta_s\ge \epsilon / q^{\widetilde{S}-s}$,
we can show that the total number of iterations performed by the modified FAPL method is bounded by
\begin{equation}
\widetilde{S}+  \left( \frac{2^{\frac{1+\rho}{2}}cM} {\theta\beta(1+\rho)\mu^{\frac{1+\rho}{2}}\epsilon^{\frac{1-\rho}{2}}}   \right)^{\frac{2}{1+3\rho}}\sum_{s=1}^{\widetilde{S}}q^{(\widetilde{S}-s)\frac{1-\rho}{1+3\rho}}.
\end{equation}
Specifically, if $f$ is smooth ($\rho=1$), then the above bound is reduced to
\begin{equation}
\widetilde{S}\left(\sqrt{\frac{2cM}{\theta\beta\mu}}+1\right).
\end{equation}
If $f$ is nonsmooth ($\rho=0$) or weakly smooth ($\rho\in (0,1)$), then the above bound is equivalent to
\begin{align}
\widetilde{S}+  \left( \frac{2^{\frac{1+\rho}{2}}cM} {\theta\beta(1+\rho)\mu^{\frac{1+\rho}{2}}\epsilon^{\frac{1-\rho}{2}}}   \right)^{\frac{2}{1+3\rho}}\sum_{s=1}^{\widetilde{S}}q^{(\widetilde{S}-s)\frac{1-\rho}{1+3\rho}}
  \le\ \widetilde{S}+  \frac{1}{1-q^{\frac{1-\rho}{1+3\rho}}}\left( \frac{2^{\frac{1+\rho}{2}}cM} {\theta\beta(1+\rho)\mu^{\frac{1+\rho}{2}}\epsilon^{\frac{1-\rho}{2}}}   \right)^{\frac{2}{1+3\rho}}.
\end{align}
\end{proof}

\vgap

Now let us consider the structured CP problems with $f$ given by  \eqref{USLpro},
where the smooth component $\hat{f}$ is strongly convex with modulus $\mu$.
Similar to the modified FAPL method, we present a modified FUSL method for solving
this strongly convex structured CP problems as follows.

\begin{procedure}[h]
\caption{\label{funFUSLsc} The modified FUSL gap reduction procedure: $(x^+,D^+, \lb^+)=\tilde{\mathcal{G}}_{FAPL}(\hat{x},D,\lb, r,\beta,\theta)$}
In Procedure \ref{funFUSL}, set $\overline{x}=\hat{x}$, and consequently the prox-function $d$ is replaced by $\|x-\hat{x}\|^2/2$.
\end{procedure}

\begin{algorithm}[h]
	\caption{\label{algFUSLsc}{The modified FUSL method for minimizing strongly convex functions}}
	In Algorithm \ref{algFUSL}, change steps \ref{stepFUSLinit},\ref{stepFUSLinitbound} and \ref{stepFUSLcall} to
	\begin{algorithmic}[1]
		\makeatletter
		\setcounter{ALG@line}{-1}
		\makeatother
		\State \label{stepFUSLinit_sc}Choose initial lower bound $\lb_1\le f^*$, initial point $p_0\in \R^n$, initial upper bound $\ub_1=f(p_0)$, prox-function $v(\cdot)$, initial guess $D_1$ on the size $D_{v,Y}$, tolerance $\epsilon>0$ and parameters $\beta,\theta \in(0,1)$.
        \State Set $\hat{x}_1=p_0$, and $s=1$.
		\makeatletter
		\setcounter{ALG@line}{2}
		\makeatother
		\State Set $(\hat{x}_{s+1},D_{s+1},\lb_{s+1})=\tilde{\mathcal{G}}_{FUSL}(\hat{x}_s, D_s,\lb_s, \sqrt{{2(f(\hat{x}_s)-\lb_s)}/{\mu}},\beta,\theta)$ and  $\ub_{s+1}=f(\hat{x}_{s+1})$.
	\end{algorithmic}
\end{algorithm}

%Note that the outer phases of FUSL method are consisted of significant phases and non-significant phases, and the non-significant phases %dynamically estimate the parameter $D_{v,Y}$ defined in \eqref{defDvY}.

In the following theorem, we describe the convergence properties of the modified FUSL method for solving \eqref{oriproblem}-\eqref{USLpro} with strongly convex smooth component $\hat{f}$.

\begin{theorem}
\label{thmStrongFUSL} Suppose that $\{\alpha_k\}_{k\ge 1}$ in procedure $\tilde{\mathcal{G}}_{FUSL}$ are chosen such that \eqref{stepsize} holds. Then we have the following statements hold for the modified FUSL method.
\begin{enumerate}[a)]
\item The total number of iterations performed by the modified FUSL method for computing an $\epsilon$-solution
of problem \eqnok{oriproblem}-\eqnok{USLpro} is bounded by
\begin{equation}
  (S_1+\widetilde{S})\left(\sqrt{\frac{2cL_{\hat{f}}}{\theta\beta\mu}}+1\right)+\frac{4\|A\|\sqrt{\tilde{D}}}{\theta\beta(1-\sqrt{q})}\sqrt{\frac{c}{\sigma_v\mu\epsilon}},
\end{equation}
where $q$ is defined in \eqref{stepsize}, $S_1$ and $\tilde{D}$ are defined in \eqref{defS1S2D}, and $\widetilde{S}$ is defined in \eqref{eqnFAPLsc_phases}.
\item In particular, if $D_{v,Y}$ is known, and set $D_1=D_{v,Y}$ at Step \ref{stepFUSLinit}, then the number of iterations performed by the modified FUSL method is reduced to
\begin{equation}
\label{eqnFUSLsc}
\overline{N}(\epsilon):= \widetilde{S}\left(\sqrt{\frac{2cL_{\hat{f}}}{\theta\beta\mu}}+1\right)+\frac{2\|A\|}{\theta\beta(1-\sqrt{q})}\sqrt{\frac{cD_{v,Y}}{\sigma_v\mu\epsilon}}.
\end{equation}
\end{enumerate}
\end{theorem}

\begin{proof}
Similarly to the discussion in Theorem \ref{thmFUSL}, we classify the non-significant and significant phases and estimates the numbers of iterations performed by each type of phases. Suppose that the set of indices of the non-significant and significant phases are  $\{m_1,m_2,\ldots,m_{s_1}\}$ and $\{n_1,n_2,\ldots,n_{s_2}\}$ respectively. Then the number of nonsignificant phases is bounded by $S_1$, i.e., $s_1\le S_1$. And since $\Delta_1=\ub_1-\lb_1$, so the number of significant phases is bounded by $\widetilde{S}$ defined above, i.e., $s_2\le S_2$. 

In view of Proposition \ref{Prop_FUSL_GAP}, and substitute $r=\sqrt{\frac{2\Delta}{\mu}}$, we have for any phase
\begin{equation}
  \widetilde{N}(\Delta,D):=\sqrt{\frac{2cL_{\hat f}}{\theta\beta\mu}}+\frac{2\|A\|}{\theta\beta}\sqrt{\frac{cD}{\sigma_v\mu\Delta}}+1.
\end{equation}
Following similar discussion in Theorem \ref{thmFUSL}, we have the number of iterations performed by non-significant phases in the modified FUSL method is bounded by
\begin{align}
\widetilde{N}_1=&\sum_{k=1}^{s_1}\widetilde{N}(\Delta_{m_k},D_{m_k})\le\sum_{k=1}^{s_1}\widetilde{N}(\epsilon,\tilde{D}/2^{s_1-k})\\
              &\le S_1\left(\sqrt{\frac{2cL_{\hat{f}}}{\theta\beta\mu}}+1\right)+\frac{2\|A\|}{\theta\beta}\sqrt{\frac{c\tilde{D}}{\sigma_v\mu\epsilon}}\sum_{k=1}^{S_1}q^{\frac{S_1-k}{2}}\\
              & \le S_1\left(\sqrt{\frac{2cL_{\hat{f}}}{\theta\beta\mu}}+1\right)+\frac{2\|A\|}{\theta\beta(1-\sqrt{q})}\sqrt{\frac{c\tilde{D}}{\sigma_v\mu\epsilon}}.
\end{align}
And the bound on number of iterations performed by all significant phases is given by
\begin{align}
\label{eqnFUSLstronglyconvex}
\widetilde{N}_2=&\sum_{k=1}^{s_2}\widetilde{N}(\Delta_{n_k},D_{n_k})\le\sum_{k=1}^{s_2}\widetilde{N}(\epsilon/q^{s_2-k},\tilde{D})\\
              &\le \widetilde{S}\left(\sqrt{\frac{2cL_{\hat{f}}}{\theta\beta\mu}}+1\right)+\frac{2\|A\|}{\theta\beta}\sqrt{\frac{c\tilde{D}}{\sigma_v\mu\epsilon}}\sum_{k=1}^{\widetilde{S}}q^{\frac{\widetilde{S}-k}{2}}\\
              & \le \widetilde{S}\left(\sqrt{\frac{2cL_{\hat{f}}}{\theta\beta\mu}}+1\right)+\frac{2\|A\|}{\theta\beta(1-\sqrt{q})}\sqrt{\frac{c\tilde{D}}{\sigma_v\mu\epsilon}}.
\end{align}
Therefore, the total number of iterations is bounded by $\widetilde{N}_1+\widetilde{N}_2$, and thus part a) holds.

For part b), in view of Lemma \ref{lemSigPhase}, we can see that if $D_1=D_{v,Y}$, then $D_s\equiv D_{v,Y}$ for all $s\ge 1$, and all phases of the modified FUSL method are significant. Therefore, replace $D_{n_k}$ and $\tilde{D}$ in \eqref{eqnFUSLstronglyconvex}, we can conclude part b) holds.
\end{proof}

\vgap

In view of the above Theorem \ref{thmStrongFUSL}, we can see that the iteration complexity of the modified FUSL method for solving
the structured CP problem \eqref{USLpro} is bounded by
$
	\cO\left(\|A\|/\sqrt{\epsilon} \right).
$

\setcounter{equation}{0}
\section{Numerical experiments}
\label{secExperiments}
In this section we present our experimental results of solving a few large-scale CP problems, including the quadratic programming problems with large Lipschitz constants, and two different types of variation based image reconstruction problems, using the FAPL and FUSL methods, and compare them with some other first-order algorithms. All the algorithms were implemented in MATLAB, Version R2011a and all experiments were performed on a desktop with an Inter Dual Core 2 Duo 3.3 GHz CPU and 8G memory.
\subsection{Quadratic programming}
\label{secQP}
The main purpose of this section is to investigate the performance of the FAPL method for solving smooth CP problems especially with large Lipschitz constants. For this purpose, we consider the quadratic programming problem:
\begin{equation}
\label{QPprob}
\min_{\|x\|\le 1}\|Ax-b\|^2,
\end{equation}
where $A\in \mathbb{R}^{m \times n}$ and $b\in \mathbb{R}^m$. We compare the FAPL method with Nesterov's optimal method (NEST) for smooth functions \cite{Nest05-1}, NERML \cite{BenNem05-1}, and APL \cite{Lan13-1}. We also compare the FAPL method with the
the built-in Matlab linear system solver in view of its good practical performance. In the APL method, the subproblems are solved by MOSEK \cite{Mosek}, an efficient software package for linear and second-order cone programming. Two cases with different choices of the initial lower bound LB in this experiments are conducted: (1). $LB=0$ and (2). $LB=-\infty$.

%We assume there is a solution $x^{\ast}$ in the Euclidean unit ball $\{x\in \mathbb{R}^n: \|x\|\le 1\}$ without lose of generality.

In our experiments, given $m$ and $n$, two types of matrix $A$ are generated. The first type of matrix $A$ is randomly generated with entries uniformly distributed in [0,1], while the entries of the second type are normally distributed according to $N(0,1)$.  We then randomly choose an optimal solution $x^{\ast}$ within the unit ball in $\mathbb{R}^n$, and generate the data $b$ by $b=Ax^{\ast}$. We apply all the four methods to solve \eqref{QPprob} with this set of data $A$ and $b$, and the accuracy of the generated solutions are measured by $e_k=\|Ax_k-b\|^2$. The results are shown in Tables \ref{QPtableone}, \ref{QPtabletwo} and \ref{QPtablethree}.

\begin{table}[!hbp]
\centering
\caption{Uniformly distributed QP instances}
\label{QPtableone}
\begin{tabular}{|l|l|lll|lll|}

\hline
\multicolumn{8}{|c|}{$A:n=4000, m=3000, L=2.0e6, e_0=2.89e4$} \\
\hline
Alg & LB & Iter.  & Time & Acc. & Iter.  & Time & Acc.\\
\hline
FAPL & 0 & 103 & 3.06 & 9.47e-7 & 142  & 3.76 & 8.65e-9 \\

 & $-\infty$ & 277  & 6.55 & 5.78e-7 & 800  & 19.18 & 2.24e-11\\
 \hline
 APL & 0 & 128 & 37.10 & 9.07e-7 & 210  & 60.85 & 9.82e-9 \\
  & $-\infty$ & 300  & 85.65 & 6.63e-6 & 800  &234.69 & 2.59e-9\\
 \hline
 NERML & 0 & 218 & 58.32 & 9.06e-7 & 500  & 134.62 & 1.63e-8 \\
 & $-\infty$ & 300  & 84.01 & 1.02e-2 & 800 & 232.14 & 1.71e-3\\
 \hline
  NEST & - & 10000 & 220.1 & 3.88e-5 & 20000  & 440.02 & 3.93e-6\\
  \hline

\end{tabular}

\begin{tabular}{|l|l|lll|lll|}

\hline
\multicolumn{8}{|c|}{$A:n=8000, m=4000, L=8.0e6, e_0=6.93e4$} \\
\hline
Alg & LB & Iter.  & Time & Acc. & Iter.  & Time & Acc.\\
\hline
FAPL & 0 & 70 & 4.67 & 7.74e-7 & 95  & 6.46 & 6.85e-10 \\

 & $-\infty$ & 149  & 8.99 & 6.27e-7 & 276  & 16.94 & 6.10e-10\\
 \hline
 APL & 0 & 79 & 71.24 & 7.79e-7 &144  & 129.52 & 3.62e-9 \\
  & $-\infty$ & 248  & 205.48 & 8.16e-7 & 416  & 358.96 & 8.68e-9\\
 \hline
 NERML & 0 & 153 & 128.71 & 7.30e-7 & 300  & 251.79 & 4.03e-9 \\
 & $-\infty$ & 300  & 257.54 & 1.18e-3 & 800 & 717.13 & 9.24e-5\\
 \hline
  NEST & - & 10000 & 681.03 & 5.34e-5 & 20000  & 1360.52 & 4.61e-6\\
  \hline

 \end{tabular}

\begin{tabular}{|l|l|lll|lll|}

\hline
\multicolumn{8}{|c|}{FAPL method for large dimension matrix} \\
\hline
Matrix A:$m\times n$ & LB & Iter.  & Time & Acc. & Iter.  & Time & Acc.\\
\hline
 $10000\times 20000$& 0 & 97 & 36.65 & 6.41e-11 & 185  & 69.31 & 7.29e-21 \\

 L=5.0e7 & $-\infty$ & 207  & 73.70 & 8.28e-8 & 800  & 292.06 & 2.32e-15\\
 \hline
 $10000\times 40000$ & 0 & 67 & 49.95 & 9.21e-11 & 122  & 91.49 & 7.27e-21 \\
 L=1.0e8 & $-\infty$ & 130  & 88.40 & 7.11e-8 & 421  & 295.15 & 1.95e-16\\
 \hline
 $10000\times 60000$ & 0 & 52 & 58.06 & 7.68e-11 & 95  & 106.14 & 8.43e-21 \\
 L=1.5e8& $-\infty$ & 156  & 160.93 & 9.84e-8 & 394 & 422.4 & 7.48e-16\\
  \hline

\end{tabular}

\end{table}

\begin{table}[!hbp]
\centering
\caption{Gaussian distributed QP instances}
\begin{tabular}{|l|l|lll|lll|}

\hline
\multicolumn{8}{|c|}{$A:n=4000, m=3000, L=2.32e4, e_0=2.03e3$} \\
\hline
Alg & LB & Iter.  & Time & Acc. & Iter.  & Time & Acc.\\
\hline
FAPL & 0 & 105 & 2.78 & 8.43e-7 & 153  & 4.10 & 7.84e-10 \\
 & $-\infty$ & 338  & 8.02 & 6.86e-7 & 696  & 16.58 & 9.74e-10\\
 \hline
 APL & 0 & 128 & 35.12 & 9.01e-7 &172  & 47.49 & 9.28e-9 \\
  & $-\infty$ &639  & 200.67 & 7.92e-7 & 800  &258.25   &1.03e-7 \\
 \hline
 NERML & 0 & 192 & 48.44 & 7.05e-7 & 276  & 70.31 & 1.09e-8 \\
 & $-\infty$ & 300  & 93.32 & 3.68e-1 & 800 & 257.25 & 6.41e-2\\
 \hline
  NEST & - & 10000 & 211.30 & 7.78e-4 & 20000  & 422.78 & 1.95e-4\\
  \hline

 \end{tabular}

\begin{tabular}{|l|l|lll|lll|}
\hline
\multicolumn{8}{|c|}{$A:n=8000, m=4000, L=2.32e4, e_0=2.03e3$} \\
\hline
Alg & LB & Iter.  & Time & Acc. & Iter.  & Time & Acc.\\
\hline
FAPL & 0 & 49 & 3.25 & 8.34e-7 & 68  & 4.37 & 7.88e-10 \\
 & $-\infty$ & 165  & 9.77 & 5.17e-7 & 280  & 16.18 & 5.06e-10\\
 \hline
 APL & 0 & 59 & 48.91 & 8.59e-7 &78  & 64.95 & 1.70e-8 \\
  & $-\infty$ & 300  & 268.47 & 9.81e-7 & 670  & 637.70 & 9.42e-10\\
 \hline
 NERML & 0 & 105 & 181.23 & 9.14e-7 & 133  & 102.68 & 1.39e-8 \\
 & $-\infty$ & 300  & 282.56 & 9.92e-3 & 800 & 760.26 & 8.32e-4\\
 \hline
  NEST & - & 10000 & 567.59 & 3.88e-4 & 20000  & 1134.38 & 9.71e-5\\
  \hline

\end{tabular}
\begin{tabular}{|l|l|lll|lll|}

\hline
\multicolumn{8}{|c|}{FAPL method for large dimension matrix} \\
\hline
Matrix A:$m\times n$ & LB & Iter.  & Time & Acc. & Iter.  & Time & Acc.\\
\hline
 $10000\times 20000$& 0 & 78 & 27.88 & 7.22e-11 & 145  & 51.81 & 6.81e-21 \\

 L=5.7e4 & $-\infty$ & 228  & 78.57 & 9.92e-8 & 800  & 280.19 & 1.37e-15\\
 \hline
 $10000\times 40000$ & 0 & 48 & 34.36 & 5.97e-11 &87  & 62.24 & 8.26e-21 \\
 L=9e4 & $-\infty$ & 156  & 106.12 & 7.18e-8 & 390  & 271.15 & 4.29e-16\\
 \hline
 $10000\times 60000$ & 0 & 34 & 36.30 & 9.88e-11 & 65  & 69.56 & 7.24e-21 \\
 L=1.2e5& $-\infty$ & 98  & 98.11 & 9.50e-8 & 350 & 361.83 & 8.34e-16\\
  \hline

\end{tabular}
\label{QPtabletwo}
\end{table}

\begin{table}
\centering
\caption{Comparison to Matlab solver}
  \begin{tabular}{|l|ll|lll|}
\hline
\multirow{ 2}{*}{Matrix A:$m\times n$} & \multicolumn{2}{|c|}{Matlab $A\backslash b$} & \multicolumn{3}{|c|}{FAPL method}\\ \cline{2-6}
    & Time & Acc. & Iter.  & Time & Acc.\\
\hline
Uniform $2000\times4000$ & 4.41 &  5.48e-24 & 204&3.59&6.76e-23\\
\hline
Uniform $2000\times6000$ & 7.12 &9.04e-24 &155&4.10&9.73e-23\\
\hline
Uniform $2000\times8000$ & 9.80  &9.46e-24  & 135 &4.45 &9.36e-23\\
\hline
Uniform $2000\times10000$ & 12.43  &1.04e-23  & 108 &4.23 &7.30e-23\\

\hline
Gaussian $3000\times5000$ & 11.17  &5.59e-25     & 207 &6.25 &7.18e-23\\
\hline
Gaussian $3000\times6000$ & 13.96  &1.43e-24     & 152 &5.50 &9.59e-23\\
\hline
Gaussian $3000\times8000$ & 19.57  &1.66e-24     & 105 &4.83 &8.17e-23\\
\hline
Gaussian $3000\times10000$ & 25.18  &1.35e-24     & 95 & 5.43 &5.81e-23\\
\hline
  \end{tabular}
  \label{QPtablethree}
\end{table}
The advantages of the FAPL method can be observed from these experiments. Firstly, it is evident that BL type methods have much less iterations than NEST especially when the Lipschitz constant is large. Among these three BL type methods, NERML requires much more iterations than APL and FAPL, which have optimal iteration complexity for this problem.

Secondly, compared with previous BL type methods (APL and NERML), FAPL has much lower computational cost for each iteration. The computational cost of FAPL method for each iteration is just slightly larger than that of NEST method. However, the cost of each iteration of APL and NERML is $10$ times larger than that of NEST.% the extra computational cost for solving the subproblems in these algorithms dominates the total computational expense.

Thirdly, consider the difference of performance for setting the lower bound to be 0 and $-\infty$, it is also evident that FAPL method is more robust to the choice of the initial lower bound and it updates the lower bound more efficiently than the other two BL methods. Though setting the lower bound to $-\infty$ increases number of iterations for all the three BL method, a close examination reveals that the difference between setting the lower bound to zero and $-\infty$ for FAPL method is not so significant as that for APL and NERML methods, especially for large matrix, for example, the second one in Table \ref{QPtableone} .

Fourthly,  FAPL needs less number of iterations than APL, especially when the required accuracy is high. A plausible explanation is that exactly solving the subproblems provides better updating for the prox-centers, and consequently, more accurate prox-centers improve the efficiency of algorithm significantly. The experiments show that, for APL and NERML, it is hard to improve the accuracy beyond $10^{-10}$. However, FAPL can keep almost the same speed for deceasing the objective value from $10^6$ to $10^{-21}$.

Finally, we can clearly see from Table \ref{QPtablethree} that FAPL is comparable to or significantly outperform the built-in Matlab solver for randomly generated linear systems, even though our code is implemented in MATLAB rather than lower-level languages, such as C or FORTRAN. We can expect that the efficiency of FAPL will be much improved by using C or FORTRAN implementation, which has been used in the MATLAB solver for linear systems.

In summary, due to its low iteration cost and effective usage of the memory of first-order information, the FAPL method is a powerful tool for solving smooth CP problems especially when the number of variables is huge and/or the value of Lipschitz constant is large.

\subsection{Total-variation based image reconstruction}
\label{sectv}
In this subsection, we apply the FUSL method to solve the non-smooth total-variation (TV) based image reconstruction problem:
\begin{equation}
\label{USLnumericalpro}
\min_{u\in \mathbb{R}^N}\frac{1}{2}\|Au-b\|_2^2+\lambda\|u\|_{TV},
\end{equation}
where $A$ is a given matrix, u is the vector form of the image to be reconstructed, $b$ represents the observed data,
and $\|\cdot\|_{TV}$ is the discrete TV semi-norm defined by
\begin{equation}
\label{TVnormdef}
\|u\|_{TV}:=\sum_{i=1}^N\|D_iu\|_2,
\end{equation}
where $D_iu\in\mathbb{R}^2$ is a discrete gradient (finite differences along the coordinate directions) of  the $i$-th component of u, and $N$ is the number of pixels in the image. The $\|u\|_{TV}$ is convex and non-smooth. 

%Many algorithms (e.g., \cite{ChHaHuPhYeYin2012-1,ChHaYaYeZh13-1,CheLanOu13-1}) are developed based on the following reformulation of \eqref{USLnumericalpro}:
%\begin{equation}
%\min_{u,w} \frac{1}{2}\|Au-b\|^2_2+\lambda\sum_{i=1}^N\|w_i\|_2, \text{ subject to  } w_i=D_i u, i=1,2,\ldots, N.
%\end{equation}
One of the approaches to solve this problem is to consider the associated dual or primal-dual formulations of \eqref{TVnormdef} based on the dual formulation of the TV norm:
\begin{equation}
\|u\|_{TV}=\max_{p\in Y} \left\langle p,Du\right\rangle , \text{where } Y=\{p=(p_1,\ldots,p_N)\in \mathbb{R}^{2N}:p_i\in \mathbb{R}^2,\|p_i\|_2\le 1,1\le i\le N  \}.
\end{equation}
Consequently, we can rewrite \eqref{USLnumericalpro} as a saddle-point problem:
\begin{equation}
\label{FUSLnumericalSP}
\min_{u\in \mathbb{R}^N}\max_{p\in Y}\frac{1}{2}\|Au-b\|^2_2+\lambda\left\langle p,Du\right\rangle .
\end{equation}
Note that \eqref{FUSLnumericalSP} is exactly the form we considered in the USL and FUSL method if we let $\hat{g}(y)=0$. Specifically, the prox-function $v(y)$ on Y is simply chosen as $v(y)=\frac{1}{2}\|y\|^2$ in these smoothing techniques.
%applying nesterov's smoothing technique\cite{Nest05-1}, the approximated smooth function $f_{\eta}(u)$ is given by:
%\begin{equation}
%f_{\eta}(u)=\frac{1}{2}\|Au-f\|^2+\lambda\max_{p\in Y}\{\left<p,Du\right>-\frac{1}{2}\|p\|^2\}.
%\end{equation}
%And the gradient of $f_{\eta}(u)$ is given by:
%\begin{equation}
%f_{\eta}'(u)=A^T(Au-f)+\lambda D^T(p^{\ast}_u),
%\end{equation}
%where $p^{\ast}_u:=\argmin_{p\in Y} {\frac{1}{2}\|p\|^2-\left<p,Du\right>}$, and $D^T$ is the adjoint operator of $D(\cdot)$.

In our experiments, we consider two types of instances depending on how the matrix $A$ is generated. Specifically, for the first case, the entries of $A$ are normally distributed, while for the second one, the entries are uniformly distributed.
For both types of instances, first, we generated the matrix $A\in \mathbb{R}^{m\times n}$, then choose some true image $x_{ture}$ and convert it to a vector, and finally compute $b$ by $b=Ax_{true}+\epsilon$, where $\epsilon$ is the Gaussian noise with distribution $\epsilon=N(0,\sigma)$. We compare the following algorithms: the accelerated primal dual (APD) method \cite{CheLanOu13-1}, Nesterov's smoothing (NEST-S) method \cite{Nest05-1,BeBoCa09-1}, and FUSL method.

For our first experiment, the matrix $A$ is randomly generated of size $4,096 \times 16,384$ with entries normally distributed according to $N(0,\sqrt{4,096})$, the image $x_{true}$ is a $128 \times 128$ Shepp-Logan phantom generated by MATLAB. Moreover, we set $\lambda=10^{-3}$ and the standard deviation  $\sigma=10^{-3}$. The Lipschitz constants are provided for APD and NEST-S, and the initial lower bound for FUSL method is set to $0$. We run $300$ iterations for all these algorithms, and report the objective value of problem \eqref{USLnumericalpro} and the relative error defined by $\|x_k-x_{true}\|_2/\|x_{true}\|_2$ as shown in  Figure \ref{figphantom}.
In our second experiment, the matrix $A$ is randomly generated with entries uniformly distributed in $[0,1]$. We use a $200 \times 200$ brain image \cite{ChHaHuPhYeYin2012-1} as the true image $x_{true}$, and set $m=20,000, \lambda=10, \sigma=10^{-2}$. Other setup is the same as the first experiment, and the results are shown in Figure \ref{figtv2}.

\begin{figure}[!htp]
 \includegraphics[width=0.5\linewidth]{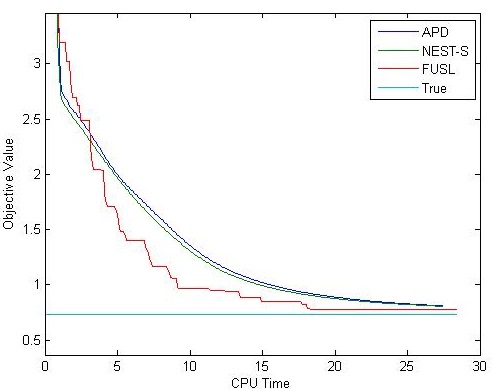}
 \includegraphics[width=0.5\linewidth]{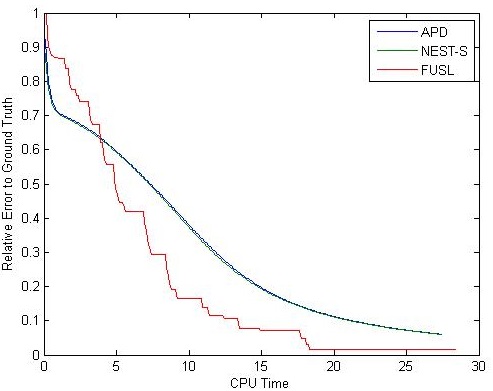}\\
  \includegraphics[width=1.0\linewidth]{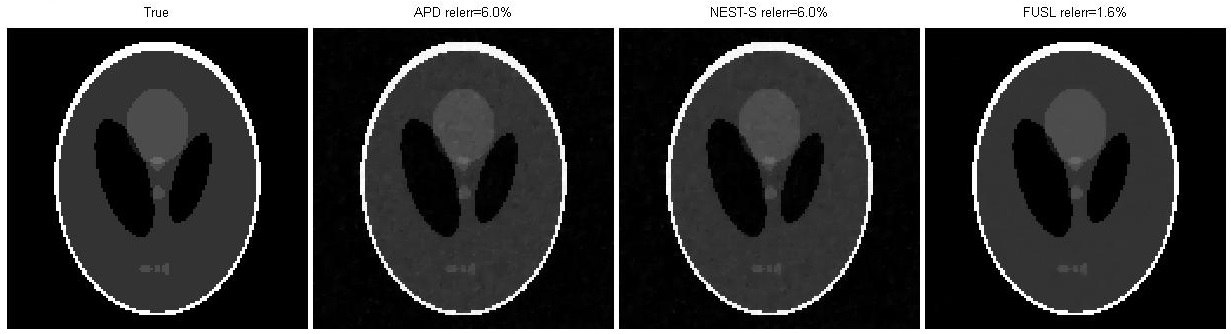}
  \caption{TV-based reconstruction (Shepp-Logan phantom)}
  \label{figphantom}
 \end{figure}

\begin{figure}[!htp]
 \includegraphics[width=0.50\linewidth]{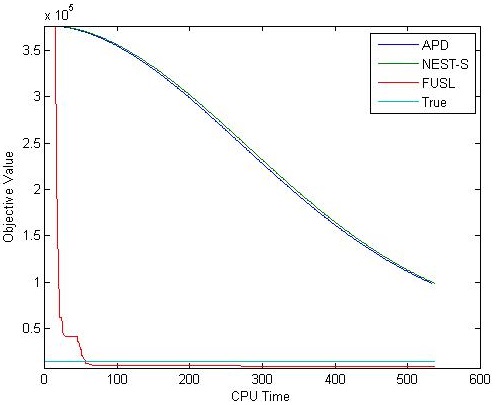}
 \includegraphics[width=0.50\linewidth]{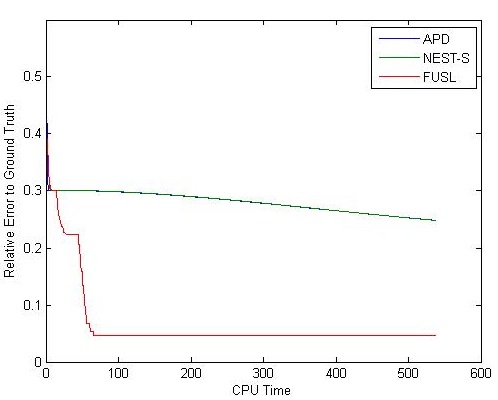}\\
 \includegraphics[width=1.0\linewidth]{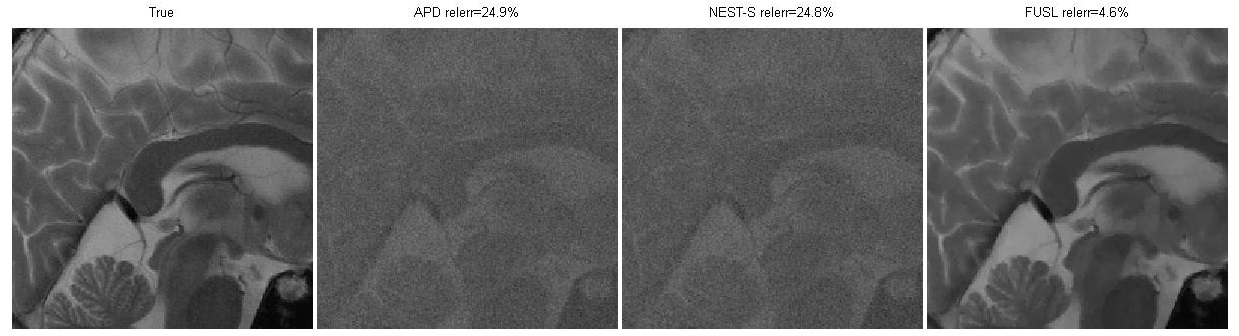}
  \caption{TV-based reconstruction (brain image)}
  \label{figtv2}
 \end{figure}

We make some observations about the results in Figures \ref{figphantom} and \ref{figtv2}. For the first experiment, there is almost no difference between APD and NEST-S method, but FUSL outperforms both of them after $5$ seconds in terms of both objective value and relative error.
%Also the computing cost of FAPL is competitive with APD and NEST-S.
The second experiment clearly demonstrates the advantage of FUSL  for solving CP problems with large Lipschitz constants. The Lipschitz constant of matrix $A$ in this instance is about $2 \times 10^8$, much larger than the Lipschitz constant (about $5.9$) in the first experiment. FUSL still converges quickly and decreases the relative error to $0.05$ in less than $100$ iterations, while APD and NEST-S converge very slowly and more than $1,000$ steps are required due to the large Lipschitz constants. It seems that FUSL is not so sensitive to the Lipschitz constants as the other two methods. This feature of FUSL makes it more efficient for solving large-scale CP problems which often have big Lipschitz constants.

In summary, for the TV-based image reconstruction problem~\eqref{USLnumericalpro}, FUSL not only enjoys the completely parameter-free property (and hence no need to estimate the Lipschitz constant), but also demonstrates significant advantages for its speed of convergence and
its solution quality in terms of relative error, especially for large-scale problems.

\subsection{Partially parallel imaging}
\label{secPPI}
In this subsection, we compare the performance of the FUSL method with several related algorithms in reconstruction of magnetic resonance (MR) images from partial parallel imaging (PPI), to further confirm the observations on advantages of this method. The detailed background and description of PPI reconstruction can be found in \cite{ChHaHuPhYeYin2012-1}. This image reconstruction problem in two dimensional cases can be modeled as
$$\min_{u\in C^n}\sum_{j=1}^k\|M\mathcal{F}S_ju-f_j\|^2+\lambda\sum_{i=1}^N\|D_iu\|_2,$$
where $u$ is the vector form of a two-dimensional image to be reconstructed, $k$ is the number of MR coils (consider them as sensors) in the parallel imaging system. $F\in C^{n\times n}$ is a 2D discrete Fourier transform matrix, $S_j\in C^{n\times n}$ is the sensitivity map of the $j$-th sensor, and $M\in R^{n\times n}$ is a binary mask describes the scanning pattern. Note that the percentages of nonzero elements in $M$ describes the compression ration of PPI scan. In our experiments, the sensitivity map $\{S_j\}_{j=1}^k$ is shown in Figure \ref{figppisensmap}, the image $x_{true}$ is of size $512 \times 512$ shown in Figures \ref{figppione} and \ref{figppitwo}, and the measurements $\{f_j\}$ are generated by
\begin{equation}
\label{PPIdatagen}
f_j=M(FS_jx_{true}+\epsilon^{re}_j/\sqrt{2}+\epsilon^{im}_j/\sqrt{-2}),\ j=1,\ldots,k,
\end{equation}
where $\epsilon^{re}_j,\epsilon^{im}_j$ are the noise with entries independently distributed according to $N(0,\sigma)$. We conduct two experiments on this data set with different acquisition rates, and compare the FUSL method to NEST-S method, and the accelerated linearized alternating direction of multipliers (AL-ADMM) with line-search method \cite{OuChLaPa14-1}.

For both experiments, set $\sigma=3\times 10^{-2},\lambda=10^{-5}$, and $\{f_j\}_{j=1}^k$ are generated by \eqref{PPIdatagen}. In the first experiment,
we use Cartesian mask with acquisition rate $14\%$: acquire image in one row for every successive seven rows, while for the second one, we use Cartesian mask with acquisition rate $10\%$: acquire image in one row for every successive ten rows. The two masks are shown in Figure \ref{figppisensmap}. The results of the first and second experiment are shown in Figures \ref{figppione} and \ref{figppitwo} respectively.
These experiments again demonstrate the advantages of the FUSL method over these state-of-the-art techniques for PPI image reconstruction,

\begin{figure}[!htp]
 \includegraphics[width=0.48\linewidth]{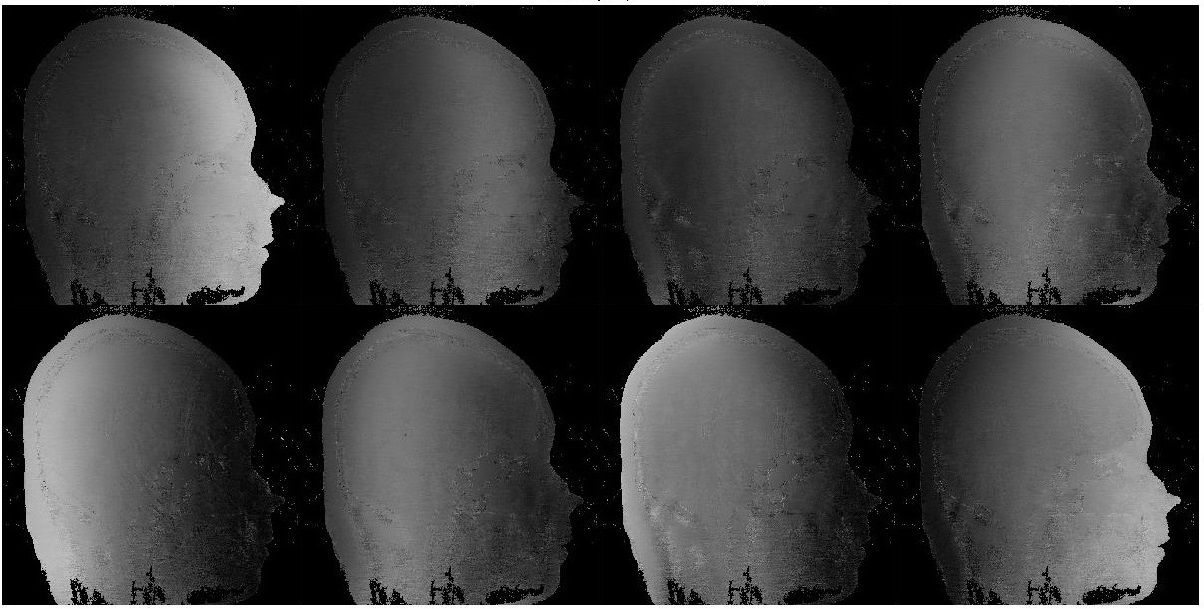}
 \includegraphics[width=0.25\linewidth]{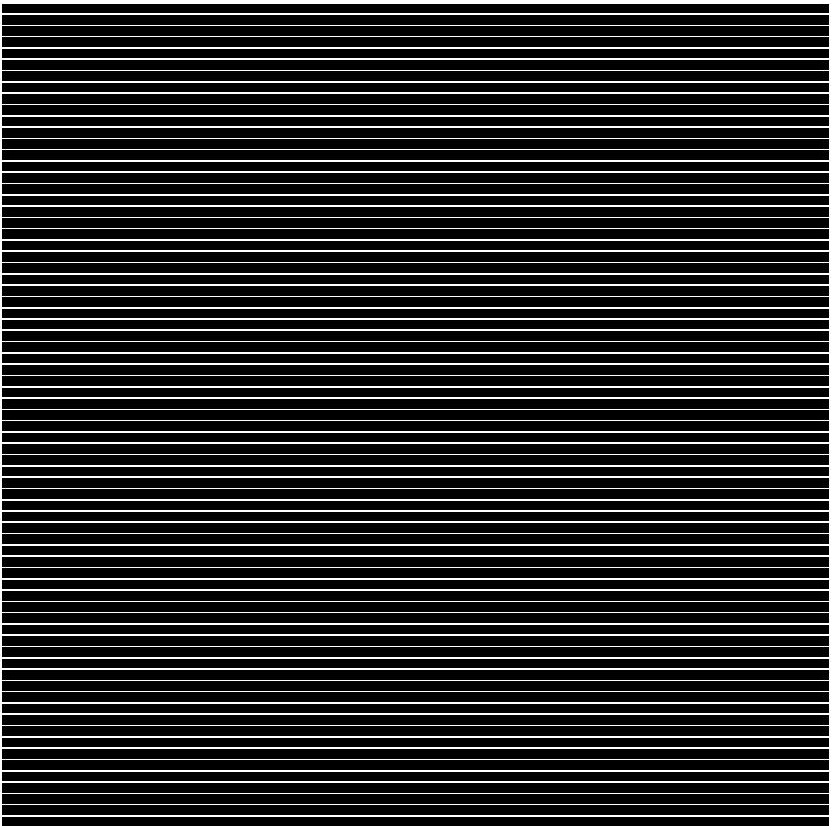}
  \includegraphics[width=0.25\linewidth]{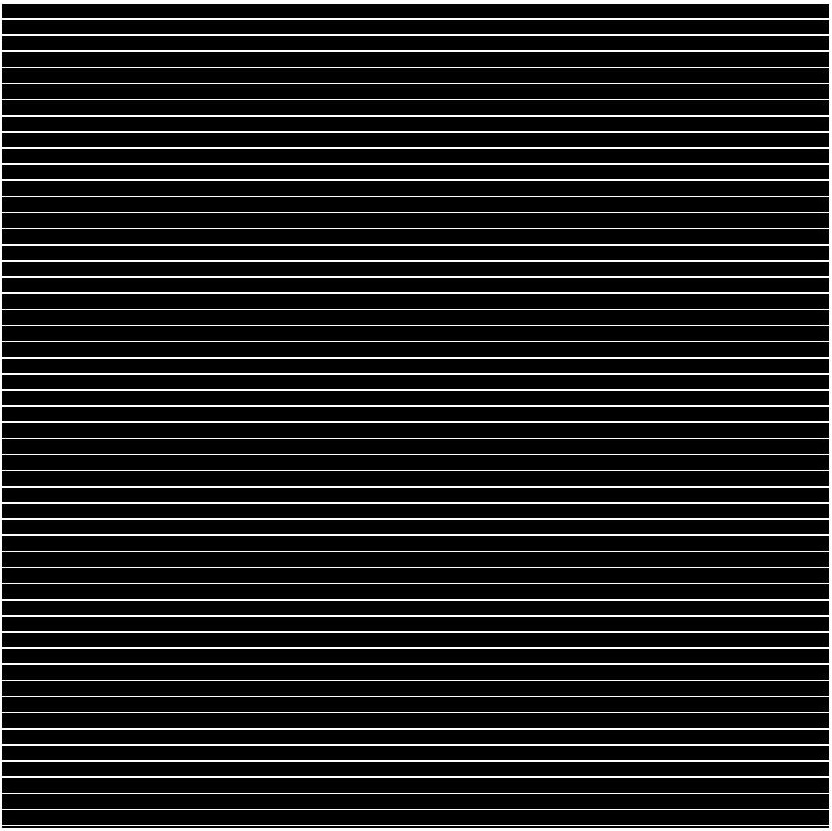}
 \caption{Sensitivity map and Cartesian mask}
 \label{figppisensmap}
 \end{figure}

\begin{figure}[!htp]
 \includegraphics[width=0.50\linewidth]{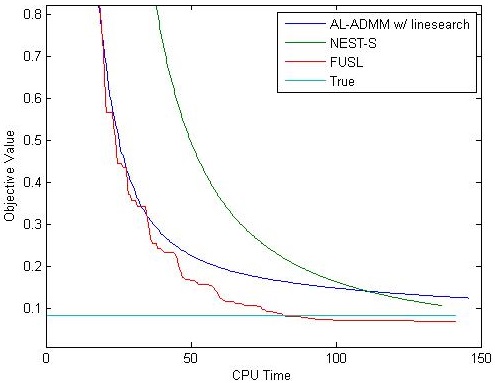}
 \includegraphics[width=0.50\linewidth]{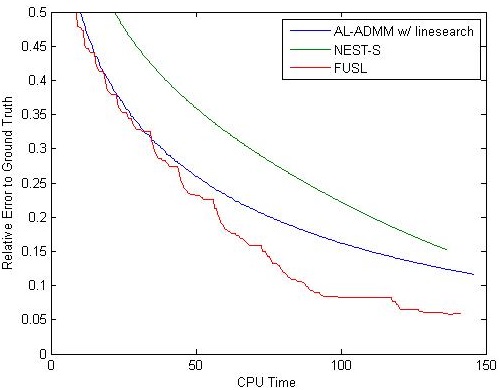}\\
  \includegraphics[width=1.0\linewidth]{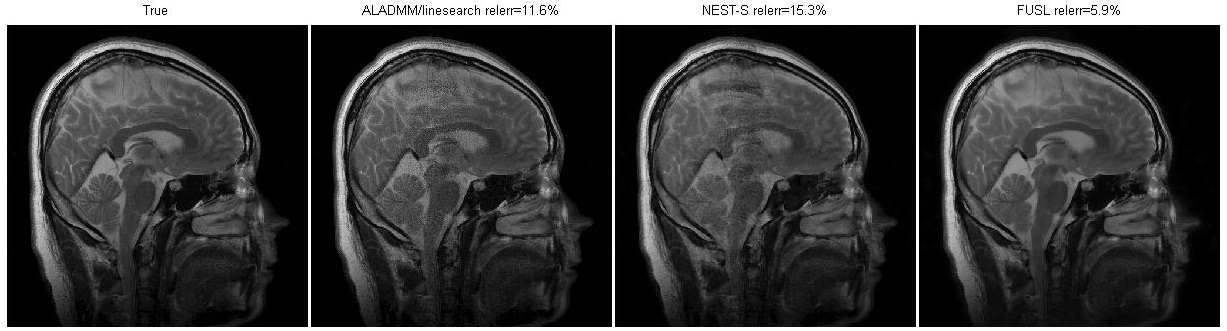}
  \caption{PPI image reconstruction (acquisition rate: $14\%$)}
  \label{figppione}
 \end{figure}

\begin{figure}[!htp]
 \includegraphics[width=0.5\linewidth]{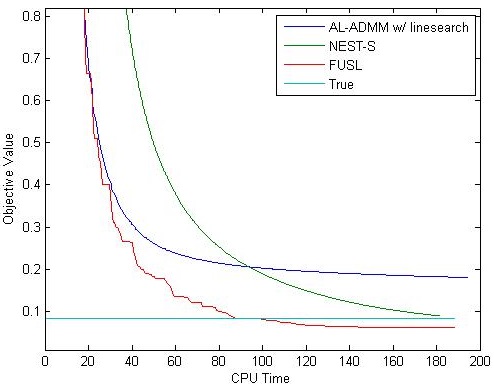}
 \includegraphics[width=0.5\linewidth]{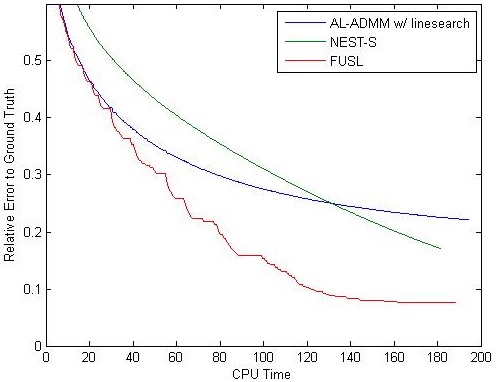}\\
 \includegraphics[width=1.0\linewidth]{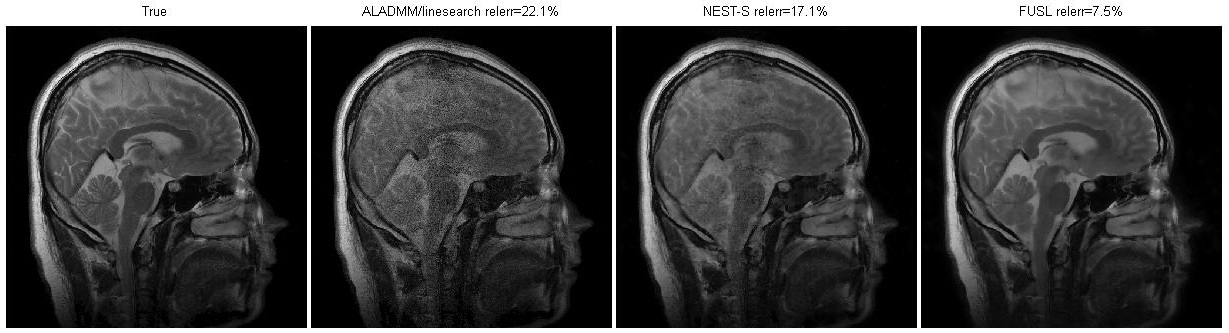}
  \caption{PPI image reconstruction (acquisition rate: $10\%$)}
  \label{figppitwo}
 \end{figure}

\section{Concluding remarks}
In this paper, we propose two new bundle-level type methods, the FAPL and FUSL methods, to uniformly solve black-box smooth, nonsmooth, and weakly smooth CP problems and a class of structured nonsmooth problems. Our methods achieve the same optimal iteration complexity and maintain all the nice features of the original APL and USL methods. Meanwhile, by simplifying the subproblems involved in the algorithms and solving them exactly, the FAPL and FUSL methods reduce the iteration cost and
increase the accuracy of solutions significantly,  and thus overcome the drawback of these existing bundle-level type methods applied to large-scale CP problems. Furthermore, by introducing a generic algorithmic framework, we extend the uniformly optimal bundle-level type methods to unconstrained problems and broaden the applicability of these algorithms.  The complexity of bundle-level type methods for unconstrained convex optimizations has been analyzed for the first time in the literature. The numerical results for least square problems and total variation based image reconstruction clearly demonstrate the advantages of FAPL and FUSL methods over the original APL, USL and some other state-of-the-art first-order methods.

%\bibliography{glan-bib}
%\bibliographystyle{plain}

\newcommand{\noopsort}[1]{} \newcommand{\printfirst}[2]{#1}
  \newcommand{\singleletter}[1]{#1} \newcommand{\switchargs}[2]{#2#1}

\end{document}